\documentclass[12pt]{article}

\usepackage{geometry}
\usepackage[normalem]{ulem}
\usepackage[dvips]{graphicx}
\usepackage{float}
\usepackage{booktabs}
\usepackage{textcomp}
\usepackage{amsthm}
\usepackage{amsmath}
\usepackage{amssymb}
\usepackage{newtxmath}
\usepackage{bbm}
\usepackage{booktabs}
\usepackage{rotating}
\usepackage[hidelinks]{hyperref}
\usepackage[backend=bibtex,maxnames=10]{biblatex}
\usepackage{color}
\usepackage{multirow}
\usepackage{tabularx}
\usepackage[all]{xy}
\usepackage{algorithm}
\usepackage{scalefnt}
\usepackage[shortlabels]{enumitem}
\usepackage{epstopdf}
\ifpdf
\DeclareGraphicsExtensions{.eps,.pdf,.png,.jpg}
\else
\DeclareGraphicsExtensions{.eps}
\fi
\graphicspath{{FIG/}}
\usepackage{subfigure}

\newtheorem{theorem}{Theorem}[section]
\newtheorem{lemma}[theorem]{Lemma}

\newtheorem{assumption}{Assumption}

\newtheorem{proposition}[theorem]{Proposition}

\newtheorem{definition}[theorem]{Definition}

\usepackage{{algcompatible}}
\newtheorem{innercustomthm}{Algorithm}
\newenvironment{customthm}[1]
{\renewcommand\theinnercustomthm{#1}\innercustomthm}
{\endinnercustomthm}

\graphicspath{{FIG/}}
\usepackage{subfigure}

\DeclareMathOperator{\cone}{cone}

\newcommand{\N}{\mathbb{N}}
\newcommand{\Q}{\mathbb{Q}}
\newcommand{\R}{\mathbb{R}}
\newcommand{\D}{\mathbb{D}}

\newcommand{\X}{\mathcal{X}}
\newcommand{\K}{\mathcal{K}}
\newcommand{\F}{\mathcal{F}}
\newcommand{\T}{\mathcal{T}}
\newcommand{\I}{\mathcal{I}}
\newcommand{\J}{\mathcal{J}}
\newcommand{\Ll}{\mathcal{L}}
\newcommand{\Dd}{\mathcal{D}}

\newcommand{\bvec}[1]{{\bf{#1}}}
\newcommand{\vecx}{{\bf{x}}}
\newcommand{\barx}{{\bar{\vecx}}}

\newcommand{\vecs}{{\bf{s}}}
\newcommand{\vecy}{{\bf{y}}}
\newcommand{\vecu}{{\bf{u}}}
\newcommand{\vecd}{{\bf{d}}}
\newcommand{\mA}{{\bf{A}}}

\newcommand{\abs}[1]{|#1|}

\newcommand{\set}[1]{\{#1\}}

\newcommand{\eps}{\eps}

\newcommand{\Disp}{\displaystyle}

\definecolor{blue}{rgb}{0.00,0.00,1.00}
\definecolor{green}{rgb}{0.059, 0.78, 0.125}
\definecolor{purple}{rgb}{0.49,0.18,0.56}
\definecolor{pink}{rgb}{1,0.75,0.796}

\newcommand{\Gint}{\mathcal{G}^{\tt int}}
\newcommand{\Omegaint}{\Omega^{\tt int}}
\newcommand{\intOmegaint}{\mathring\Omega^{\tt int}}
\newcommand{\Gext}{\mathcal{G}^{\tt ext}}
\newcommand{\Omegaext}{\Omega^{\tt ext}}

\newcommand{\gmink}{g^{\min}_k}

\begin{document}
\title{Nonlinear Derivative-free Constrained Optimization with a Penalty-Interior Point Method and Direct Search}
\author{\small Andrea Brilli \thanks{``Sapienza'' University of Rome, Department of Computer Control and Management Engineering ``A. Ruberti'', Rome, Italy \texttt{(brilli@diag.uniroma1.it)}} \and \small Ana L. Custódio \thanks{NOVA School of Science and Technology, Center for Mathematics and Applications (NOVA MATH), Campus de Caparica, 2829-516 Caparica, Portugal \texttt{(alcustodio@fct.unl.pt)}. This work was funded by national funds through FCT - Fundação para a Ciência e a Tecnologia I.P., under the scope of projects UIDP/00297/2020 and UIDB/00297/2020 (Center for Mathematics and Applications).} \and \small Giampaolo Liuzzi\thanks{``Sapienza'' University of Rome, Department of Computer Control and Management Engineering ``A. Ruberti'', Rome, Italy \texttt{liuzzi@diag.uniroma1.it}} \and \small Everton J. Silva\thanks{NOVA School of Science and Technology, Center for Mathematics and Applications (NOVA MATH), Campus de Caparica, 2829-516 Caparica, \texttt{(ejo.silva@campus.fct.unl.pt)}.  This work was funded by national funds through FCT - Fundação para a Ciência e a Tecnologia I.P., under the scope of projects UI/BD/151246/2021, UIDP/00297/2020, and UIDB/00297/2020 (Center for Mathematics and Applications).}}
\date{}
\maketitle


\begin{abstract}
In this work, the joint use of a mixed penalty-interior point method and direct search is proposed, to address {nonlinear} constrained derivative-free optimization problems. A  merit function is considered, wherein the set of nonlinear inequality constraints is divided into two groups: one treated with a logarithmic barrier approach, and another, along with the equality constraints, addressed using a penalization term. This strategy is adapted and incorporated into a direct search method, enabling the effective handling of general nonlinear constraints. Convergence to KKT-stationary points is established under continuous differentiability assumptions, without requiring any kind of convexity. Computational experiments on analytical problems and an engineering application demonstrate the robustness, efficiency, and overall effectiveness of the proposed method, when compared with state-of-the-art solvers.
\end{abstract}


\maketitle

\section{Introduction}
In this work, the following derivative-free optimization problem with general constraints (linear and nonlinear) is considered:
\begin{equation}
\begin{array}{l}
\min\ f(\vecx)\\
\text{s.t.}\ \ g_\ell(\vecx)\leq 0 \quad \ell=1,\dots,m,\\
\hspace{0.7cm} h_j(\vecx) = 0 \quad j=1,\dots,p,\\
\hspace{0.7cm} \vecx \in \X,\\
\end{array}
\label{P}
\tag{P}
\end{equation}
where $f:\X\to\R$, $g_\ell:\X\to\R$, $\ell=1,\dots,m$, $h_j:\X\to\R$, $j=1,\dots,p$ are continuously differentiable in $\X$, and $\X=\set{\vecx\in\R^n\mid \mA\vecx\leq \bvec{b}}$, with $\ \mA\in\R^{q\times n}$ and $\bvec{b}\in \R^q$. {Also}, it is assumed that the derivatives of the functions $f$, $g_\ell$, and $h_j$ can be neither calculated nor explicitly approximated. This is a common setting in a context of simulation-based optimization, where function evaluations are computed through complex and expensive computational simulations~\cite{Audet_Hare_2017,Conn_Scheinberg_Vicente_2009}.

Constrained derivative-free optimization is not new. In the context of pattern search methods, the early works are  co-authored by Lewis and Torczon, first considering bound~\cite{Lewis_Torczon_1999} or {linear} constrained problems~\cite{Lewis_Torczon_2000}, and after for general nonlinear constraints~\cite{Lewis_Torczon_2002}. When the constraints are only linear, inspired by the work of May~\cite{May_1974}, procedures were developed that allow to conform the directions to be used by the algorithms to the geometry of the feasible region imposed by the nearby constraints~\cite{Kolda_Lewis_Torczon_2007,Lewis_Torczon_2000,Lucidi_Sciandrone_Tseng_2002}, including specific strategies to address the degenerate case~\cite{Abramson_Brezhneva_DennisJr_Pingel_2008}. In the presence of nonlinear constraints, augmented Lagrangian approaches have been proposed~\cite{Lewis_Torczon_2002}, reducing the problem solution to a sequence of bound constrained minimizations of an augmented Lagrangian function. A numerical study that also employs an augmented Lagrangian approach is given in~\cite{Gramacy_Gray_LeDigabel_Lee_Ranjan_Wells_Wild_2016}.

In the original presentation of Mesh Adaptive Direct Search (MADS)~\cite{Audet_DennisJr_2006}, a generalization of pattern search methods, constraints were addressed with an extreme barrier approach, only evaluating feasible points. If this saves in function evaluations, a very relevant feature for the target problem class, it does not take advantage on the local information obtained about the feasible region. Following the filter approaches proposed for derivative-based optimization~\cite{Fletcher_Leyffer_2002} and already explored in pattern search methods~\cite{Audet_DennisJr_2004}, linear and nonlinear inequalities started to be treated in MADS with the progressive barrier technique~\cite{Audet_DennisJr_2009}. Later, the approach was extended to linear equality constraints~\cite{Audet_LeDigabel_Peyrega_2015}, by reformulating the optimization problem, thereby reducing the number of original variables.

For directional direct search methods using linesearch, other approaches have been taken to address general nonlinear constraints. These include nonsmooth exact penalty functions~\cite{Liuzzi_Lucidi_2009}, where the original {nonlinear} constrained problem is converted into the unconstrained or {linear} constrained minimization of a nonsmooth exact penalty function. To overcome the limitations associated to this approach, in~\cite{Liuzzi_Lucidi_Sciandrone_2010} a sequential penalty approach has been studied, based on the smoothing of the nondifferentiable exact penalty function, including a well-defined strategy for updating the penalty parameter.  Recently, in the same algorithmic framework, Brilli \emph{et al.}~\cite{Brilli_Liuzzi_Lucidi_2021} proposed the use of a merit function that handles inequality constraints by means of a logarithmic barrier approach and equality constraints by considering a penalization term. This approach allows an easy management of relaxable and unrelaxable constraints and avoids the abrupt discontinuities introduced by the extreme barrier approach.

\subsection{Our contribution}
In this work, the strategy of~\cite{Brilli_Liuzzi_Lucidi_2021} is adapted and incorporated into generalized pattern search (GPS). Departing from SID-PSM~\cite{Custodio_Rocha_Vicente_2010,Custodio_Vicente_2007}, an implementation of a GPS method that uses polynomial models both at the search and the poll steps to improve the numerical performance of the code, LOG-DS is developed, a direct search method able to explicitly address nonlinear constraints by a mixed penalty-logarithmic barrier approach. Section~\ref{sid_psm_log} proposes the algorithmic structure. Convergence is established in Section~\ref{convergence_analysis}. Details of the numerical implementation are provided in Section~\ref{sec:implementation} and numerical results are reported in Section~\ref{numerical_experiments}, comparing the numerical performance of LOG-DS with state-of-the-art solvers. Section~\ref{conclusions} summarizes our conclusions. An appendix completes the paper, including some auxiliary technical results.

\subsection{Notation}
Throughout this paper, vectors will be written in lowercase boldface (e.g., ${\bf{v}} \in \R^n$, $n \geq 2$) while matrices will be written in uppercase boldface (e.g., ${\bf{S}}\in \R^{n\times p}$). The set of column vectors of a matrix $\bf{D}$ will be denoted by $\D$, and more generally sets such as $\N, \Q, \R$ will be denoted by blackboard letters. The set of nonnegative real numbers will be denoted by $\R_+$. Sequences indexed in $\N$ will be denoted by $\{a_k\}_{k\in\N}$ or $\{a_k\}$ in absence of ambiguity. Given a point $\vecx\in\R^n$ and a set $\Omega \subset \R^n$, the notation $\T_{\Omega}(\bf{x})$ is used to denote the tangent cone to $\Omega$ computed at $\vecx$, $\mathring{\Omega}$ to denote the interior of $\Omega$, and $\partial \Omega$ to represent its boundary.

\par\smallskip

\section{A direct search algorithm for nonlinear constrained optimization}
\label{sid_psm_log}

This section is devoted to the introduction of a new directional direct search algorithm, based on the sequential minimization of an adequate merit function, to solve {nonlinear} constrained derivative-free optimization problems. The first part of the section details the proposed algorithmic structure. The second part analyzes the role of linear constraints and their relationship with the choice of directions to be explored in the poll step. Finally, the main assumptions required for the theoretical analysis are stated and some preliminary results are established.

\subsection{LOG-DS algorithm}
The strategy proposed to solve Problem~\eqref{P} builds upon the original idea of~\cite{Brilli_Liuzzi_Lucidi_2021}, where a merit function is defined which guarantees strict feasibility of general inequality constraints by means of logarithmic barrier penalty terms. Equality constraints are treated by exterior penalty terms. The weakness of such an approach is that one needs an initial guess which strictly satisfies all the inequality constraints, {so that the merit function is finite}. In the present work, this requirement is used to partition the original set of nonlinear inequalities into two subsets, such that feasibility is always guaranteed for the initialization. To this aim, given an initial point $\vecx_0$, the index set of the nonlinear inequality constraints, namely $\{1,2,\dots,m\}$, is partitioned into the following two index sets:
\[
\begin{split}
\Gint & = \{\ell\in\{1,\ldots,m\}\mid g_\ell(\vecx_0) < 0\},\\
\Gext & = \{\ell\in\{1,\ldots,m\}\mid g_\ell(\vecx_0) \geq 0\}.
\end{split}
\]
Note that $\Gint$ and $\Gext$ depend only on the initial point $\vecx_0$ and are indeed a partition of the index set of nonlinear inequality constraints, i.e. $\Gint\cup\Gext=\set{1,\dots,m}$ and $\Gint\cap\Gext=\emptyset$. 

The inequalities $g_\ell(\vecx)\leq 0$ with $\ell\in\Gint$ will be treated with a logarithmic barrier approach, whereas the ones with $\ell\in\Gext$ will be aggregated into an exterior penalization term, along with the equality constraints. The constraints associated with $\Gext$ as well as the equality constraints can be regarded as relaxable, and are only required to be satisfied at the problem solution.  
Using the index sets $\Gint$ and $\Gext$, one can define:
\begin{eqnarray*}
\Omegaint &=& \set{\vecx\in\R^n\mid g_\ell(\vecx)\leq 0,\ \ell\in\Gint},\\
\Omegaext &=& \{\vecx\in\R^n\mid g_\ell(\vecx)\leq 0,\
\ell\in\Gext \text{ and} \ h_j(\vecx)=0,\ j=1,\dots, p \}.
\end{eqnarray*}
Thus, the feasible region $\F$ of Problem~\eqref{P} is given by $\F = \X\cap \Omegaint\cap \Omegaext.$ Note that, when  all inequality constraints are either violated or active at the initial guess, then $\Gint=\emptyset$ and $\Gext=\set{1,\dots,m}$. In this case, $\Omegaint = \intOmegaint = \R^n$. However, for the sake of clarity, it is assumed that $\Gint\neq\emptyset$, i.e., at least one constraint is strictly satisfied. 

The merit function proposed in this work is defined as:
\begin{equation}
Z(\vecx;\rho) =
\begin{cases}
\begin{array}{l}
f(\vecx) - \rho\sum_{\ell\in \Gint} \log(-g_\ell(\vecx))  \\
\phantom{f(\vecx)} + \frac{1}{\rho^{\nu-1}} \sum_{\ell\in \Gext}\left(\max\{g_\ell(\vecx),0\}\right)^{\nu}\\
\phantom{f(\vecx)}+ \frac{1}{\rho^{\nu-1}}\sum_{j=1}^{p} |h_j(\vecx)|^\nu
\end{array} & \text{if}\ \vecx\in\X\ \text{and}\ g_\ell(\vecx)<0, \ell\in\Gint\\
+\infty & \text{otherwise,}
\end{cases}
\label{merit_function}
\end{equation}
where $\rho>0$ is a \textit{penalty-barrier} parameter and $\nu\in(1,2]$ plays the role of a \textit{smoothing exponent}. Note that $\rho$ is the same for the different terms of the merit function and the range of $\nu$  results from theoretical requirements (see Theorem~\ref{theo:bounded_multipliers}).


Regarding the linear constraints, they are addressed directly during the minimization of $Z(\cdot;\rho)$. Therefore, given a penalty-barrier-parameter $\rho_k$, the following problem is considered
\begin{equation}
\begin{array}{l}
\min\ Z(\vecx;\rho_k)\\
\text{s.t.}\ \vecx\in \X.\\
\end{array}
\label{P2}
\tag{\mbox{$P_{\rho_k}$}}
\end{equation}
Following the classic approach of penalty and interior point methods, the proposed strategy generates a monotone sequence of penalty-barrier parameters $\set{\rho_k}$, and sequentially solves the related linearly constrained subproblems, aiming to obtain the solution of the original Problem~\eqref{P}. Note that, in order for that to happen, one needs the sequence $\set{\rho_k}$ to converge to zero, so that, in the limit, the merit function $Z(\cdot;\rho_k)$ {coincides} with the objective function $f(\cdot)$ among feasible points. 

Taking into account the black-box nature of the functions defining the problem, a directional direct search approach is proposed to solve the subproblems~\eqref{P2}. In particular, as mentioned above, an adaptation of the SID-PSM algorithm is considered, i.e., an implementation of GPS. In the literature, GPS algorithms, when applied to the minimization of a function $\varphi:\R^n\to\R$, accept new points whenever a \textit{simple decrease} of $\varphi(\cdot)$ is obtained, that is, given a current solution $\vecx_k$ and a trial point $\vecy$, whenever $\varphi(\vecy)<\varphi(\vecx_k)$ holds. This forces the method to rely on an implicit mesh to ensure the existence of limit points of the sequence of iterates $\set{\vecx_k}$. In the proposed approach, instead, the acceptance of new points relies on the notion of \textit{sufficient decrease} \cite{Deleone_Gaudioso_Grippo_1984}. The next definition adjusts the concept of \textit{forcing function} (see~\cite{Kolda_Lewis_Torczon_2003}) and is used to define the sufficient decrease condition required to accept new points. 

\begin{definition}
\label{def:forcing_function}
Let $\xi: [0,+\infty)\to [0,+\infty)$ be a continuous and nondecreasing function. The function $\xi$ is said to be a forcing function if $\xi(t)/t\to 0$ as $t\downarrow 0$ and $\xi(t)\to 0$ implies $t\to 0$.
\end{definition}
A classic example of a forcing function is $\xi(t)=\gamma t^2$, with $\gamma>0$, which is also the most commonly used in practice.\\

\noindent The proposed algorithmic structure is detailed below.

\noindent\framebox[\textwidth]{\parbox{0.95\textwidth}{
\par
\begin{customthm}{LOG-DS}\label{LOG-DS}
\end{customthm}
\par\medskip

{\bf Data.} {$\vecx_0\in \X$, $\Gint\neq\emptyset$ and $\Gext$ such that $\vecx_0\in\intOmegaint$}, ${\cal D}$ a set of sets of normalized directions, a initial step-size $\alpha_0 > 0$, a initial penalty-barrier parameter $\rho_0>0$, a smoothing exponent $\nu\in(1,2]$, step-size and penalty-barrier contraction parameters $\theta_{\alpha},\theta_{\rho}\in (0,1)$, a step-size expansion parameter $\phi\geq 1$, and $\beta>1$. 

\par\medskip

{\bf For} $k = 0,1,2,\dots$ {\bf do}
\begin{itemize}
\item[] {\bf Step 1. (Search Step, optional)}\smallskip\\
\hspace*{1.cm} {\bf If} 
$\vecs_k \in \X$ can be computed such that $Z(\vecs_k;\rho_k)\leq Z(\vecx_k;\rho_k)-\xi(\alpha_k)$
\hspace*{1.5cm} set $\vecx_{k+1}=\vecs_k$, $\alpha_{k+1} =\phi \alpha_k$,  $\rho_{k+1} =\rho_k$, declare the iteration\\
\hspace*{1.5cm} {\bf Successful}, {and skip {\bf Steps 2.} and {\bf 3.}}\\

\item[] {\bf Step 2. (Poll Step)}\smallskip \\
\hspace*{1.cm} Select a set of normalized directions $\Dd_k\in {\cal D}$ \\
\hspace*{1.cm} {\bf If} $\exists\ \vecd_k^i\in \Dd_k$: $\vecx_k+\alpha_k \vecd_k^i\in \X$ and $Z(\vecx_k+\alpha_k \vecd_k^i;\rho_k)\le Z(\vecx_k;\rho_k)-\xi(\alpha_k)$ \\
\hspace*{1.5cm} set $\vecx_{k+1} = \vecx_k + \alpha_k \vecd_k^i$, $\alpha_{k+1}=\phi \alpha_k$,  $\rho_{k+1} =\rho_k$, declare the iteration\\ 
\hspace*{1.5cm} {\bf Successful}, {and skip {\bf Step 3.}}\\
\hspace*{1.cm} {\bf Else} set $\vecx_{k+1} = \vecx_k$, $\alpha_{k+1} = \theta_{\alpha} \alpha_k$, {and} declare the iteration\\\hspace*{1.5cm}  {\bf Unsuccessful}.\\

\item[] {\bf Step 3. (Penalty-Barrier Update Step)}\medskip \\
\hspace*{1.cm} Set $\gmink = \Disp\min_{\ell\in\Gint}\{|g_\ell(\vecx_{k})|\}$\smallskip \\
\hspace*{1.cm} {\bf If} $\alpha_{k+1} \le\min\{
\rho_k^{\beta},(\gmink)^2\}$
{\bf then} set $\rho_{k+1} =\theta_{\rho}\rho_k$.\\
\hspace*{1.cm} {\bf Else} set $\rho_{k+1} = \rho_k$.
\par\medskip
\end{itemize}
{\bf EndFor}
\par
\noindent }}
\par\bigskip

In the above scheme, {the partition sets $\Gint$ and $\Gext$ are chosen so that the initial guess $\vecx_0$ satisfies $\vecx_0\in\X\cap\intOmegaint$, i.e. $\vecx_0\in\X$ and 
$g_\ell(\vecx_0)<0$ for all $\ell\in\Gint$.} This ensures that the initial value of the merit function $Z(\vecx_0;\rho_0)$ is finite. {Note that, when $\Gint=\emptyset$, the method becomes equivalent to a sequential penalty approach, for which a theoretical analysis can be found in~\cite{Liuzzi_Lucidi_Sciandrone_2010} for linesearch-type algorithms. For the sake of brevity we omit the analysis of this case. Nevertheless, we highlight that only minor adjustments would be required}. Finally, note that $\Gint$ and $\Gext$ are determined during the initialization phase and are not modified as the algorithm progresses. The method iterates over three steps. The first two follow the general structure proposed by Audet and Dennis~\cite{Audet_DennisJr_2003} for GPS, whereas the third is related to the novel penalty-interior point approach.

\textbf{Step 1}, the \textit{search step}, allowing the use of heuristics, is very flexible, not even requiring the projection of the generated points in some type of implicit mesh, since a sufficient decrease condition is used for the acceptance of new iterates. As it is detailed in Section~\ref{sec:implementation},  the original SID-PSM algorithm uses quadratic polynomial models, which are minimized to generate new trial points. This approach is adapted into LOG-DS as described in Subsection~\ref{subsec:implementation_details}, but since it is not relevant for establishing convergence properties, for now it is omitted. If the search step produces a trial point $\vecs_k\in\X$ that provides a sufficient decrease in the merit function with respect to the current solution $\vecx_k$,  i.e., $Z(\vecs_k;\rho_k)\leq Z(\vecx_k;\rho_k)-\xi(\alpha_k)$, then the algorithm sets $\vecx_{k+1}=\vecs_k$ and proceeds with iteration $k+1$. Otherwise, the poll step is invoked.

\textbf{Step 2}, the \textit{poll step} initializes with the selection of a set of normalized \textit{poll directions} $\Dd_k$. Further considerations on the choice of $\Dd_k$ are discussed in Subsection~\ref{sec:poll_linear}. If a direction $\vecd_k^i\in\Dd_k$ is found such that the trial point $\vecx_k+\alpha_k\vecd_k^i$ belongs to $\X$ and achieves a sufficient decrease in the merit function, then the algorithm sets $\vecx_{k+1}=\vecx_{k}+\alpha_k\vecd_k^i$ and proceeds with iteration $k+1$.

Note that whenever $\vecx_{k+1}\neq\vecx_k$ one has $Z(\vecx_{k+1};\rho_k)\leq Z(\vecx_k;\rho_k)-\xi(\alpha_k)$, meaning that, during either the search step or the poll step, the method identifies a better point as candidate to the solution of~\eqref{P2}. Such iterations are called \textit{successful}, and the index set of successful iterations is denoted by $\K_s$. On the other hand, if $\vecx_{k+1}=\vecx_k$, it means that none of the trial points associated with the poll directions achieves a sufficient decrease in the merit function, i.e. $Z(\vecx_k+\alpha_k\vecd_k^i;\rho_k)>Z(\vecx_k;\rho_k)-\xi(\alpha_k)$ for all $\vecd_k^i\in\Dd_k$. Such iterations are called \textit{unsuccessful}, and the index set of unsuccessful iterations is denoted by $\K_u$.

\textbf{Step 3} is only invoked at unsuccessful iterations. This step concerns the update of the penalty-barrier parameter $\rho_k$, assessing if~\eqref{P2} can be considered as solved. Since providing an exact solution of the problem is, in general, impractical, one has to define a condition to determine if the precision of the solution $\vecx_k$ is acceptable.  The following criterion is proposed $$\alpha_{k+1} \leq \min\set{\rho_k^\beta,(\gmink)^2},$$ with $\beta>1$. It is well-known that $\alpha_k$ represents a measure of stationarity for problem \eqref{P2} (see~\cite{Kolda_Lewis_Torczon_2003}). Hence, having $\alpha_{k+1}\leq \rho_k^\beta$ implies that good approximations of the solution of~\eqref{P2} {will be sought for small values of $\rho_k$}. Additionally, $\gmink = \min_{\ell\in\Gint}\set{\abs{g_\ell(\vecx_k)}}$ represents how close $\vecx_k$ is to the boundary of $\Omegaint$. Since the merit function $Z(\vecx;\rho_k)$ takes the value infinite for all $\vecx\notin\intOmegaint$, by imposing $\alpha_{k+1}\leq (\gmink)^2$, the algorithm is forced to properly explore the interior of $\Omegaint$ before moving on to the next subproblem. For further details see~\cite{Brilli_Liuzzi_Lucidi_2021}.

\subsection{Poll directions and linear inequalities}
\label{sec:poll_linear}

In unconstrained minimization problems, the poll directions should be able to generate $\mathbb{R}^n$ through nonnegative linear combinations, in order to capture the behavior of the objective function (see~\cite{Coope_Price_2002,Hare_TarryBolduc_2020,Hare_JarryBolduc_Planiden_2023} for further details on the topic). In the  presence of linear inequalities, though, the directions must be adapted to the geometry of the set~$\X$, identified by the tangent cone. Let us formalize it in the following definition.

\begin{definition}[Active constraints and tangent related sets]
\label{active_set} For every $\vecx\in \X$, i.e., such that
$\mA\vecx\leq \bvec{b}$, let $a_i^{\top},i\in\{1,\ldots,q\}$
be the $i$th row of {$\mA$}. Define
$$\T_\X(\vecx) = \{\vecd\in\R^n \mid \bvec{a}_i^\top \vecd \leq 0,\ i\in \I_\X(\vecx)\}\quad (\text{tangent cone to } \X \text{at } \vecx),$$ where $\I_\X(\vecx) = \{i\in\set{1,\ldots,q}\mid \bvec{a}_i^\top \vecx = b_i\}$ is the set of indices of active constraints.
\end{definition}

Since the polytope $\X$ is convex, $\vecx+t\vecd\in\X$ for all $\vecx\in\X$, $\vecd\in\T_\X(\vecx)$, and $t>0$ sufficiently small. Clearly, if $\vecx$ lies in the interior of $\X$, then $\T_\X(\vecx)=\R^n$.

For any incumbent $\vecx_k$, the tangent cone at $\vecx_k$ might be used to understand the geometry of $\X$ at $\vecx_k$ and to build the poll directions accordingly. While this approach might seem reasonable, in general, it is not sufficient. Indeed, the algorithm might generate a sequence of iterates $\set{\vecx_k}$ lying in the interior of $\X$ and converging to a point $\barx\in \partial\X$. In such case, the tangent cones $\T_\X(\vecx_k)$ are all equal to $\R^n$, and the algorithm might not be able to sample the space using directions in $\T_\X(\barx)\subset \R^n$. Thus, given an iterate $\vecx_k\in \X$, it is important to be able to capture the geometry of the set $\X$ \textit{near} $\vecx_k$. In order to do that, the tangent cone is approximated by the $\varepsilon$-tangent cone~\cite{Gratton_Royer_Vicente_Zhang_2019,Kolda_Lewis_Torczon_2003,Kolda_Lewis_Torczon_2007}.

\begin{definition}[$\varepsilon$-Active constraints and tangent related sets]\label{eps_active_set}
Let $\vecx_k\in \X$, i.e., such that
$\mA\vecx_k\leq \bvec{b}$, let $a_i^{\top},i\in\{1,\ldots,q\}$
be the $i$th row of {$\mA$}. Define: $$\T_\X(\vecx_k,\varepsilon) = \{\vecd\in\R^n \mid \bvec{a}_i^\top \vecd \leq 0,\ i\in \I_\X(\vecx_k,\varepsilon)\} \quad {(\varepsilon-\text{tangent cone to } \X  \text{at }} \vecx_k),$$
where $\I_\X(\vecx_k,\varepsilon) = \{i\in\set{1,\ldots,q}\mid \bvec{a}_i^\top \vecx_k \geq b_i -\varepsilon\}$ is the set of indices of $\varepsilon$-active constraints.
\end{definition}

A relation between the tangent cone and the $\varepsilon$-tangent cone has been established in~\cite{Lucidi_Sciandrone_Tseng_2002}. Let us recall it in the following proposition.

\begin{proposition}\label{prop:epsaccurate_estimate} Let $\{\vecx_k\}_{k\in\mathbb{N}}$ be a sequence of points in $\X$, converging to $\vecx^*\in \X$. Then, there exists an $\varepsilon^*>0$ (depending only on $\vecx^*$) such that for any $\varepsilon\in (0,\varepsilon^*]$ there exists $k_\varepsilon\in\mathbb{N}$ such that
\[
\begin{split}
\I_\X(\vecx^*) & = \I_\X(\vecx_k,\varepsilon) \\
\T_\X(\vecx^*) & = \T_\X(\vecx_k,\varepsilon)
\end{split}
\]
for all $k\geq k_\varepsilon$.
\end{proposition}

Proposition~\ref{prop:epsaccurate_estimate} ensures that, if at each iteration $k$, the $\varepsilon$-tangent cone of $\vecx_k$ is used to build the set of poll directions $\Dd_k$, then, for small enough $\varepsilon>0$, the algorithm is able to capture the geometry of $\X$ at the limit point $\vecx^*$. The requirements on the sets of poll directions $\Dd_k$ used by the algorithm are formalized in Assumption~\ref{condition:geometry_ass} (see~\cite[Assumption 2]{Liuzzi_Lucidi_Sciandrone_2006}).
\begin{assumption}
\label{condition:geometry_ass} Let $\{\vecx_k\}$ be a sequence of
points in $\X$. The sequence $\{\Dd_k\}$ of poll directions
satisfies:
\[ \Dd_k = \{\vecd_k^i \mid \|\vecd_k^i\| = 1, i=1,\dots,|\Dd_k|\} \]
and for some $\bar\varepsilon > 0$, \[ \cone(\Dd_k\cap \T_\X(\vecx_k,\varepsilon)) =
\T_\X(\vecx_k,\varepsilon),\quad\forall\ \varepsilon\in(0,\bar\varepsilon].\]
Furthermore, ${\cal D}=\Disp\bigcup_{k=0}^{+\infty} \Dd_k$ is a finite set, and $|\Dd_k|$ is uniformly bounded.
\end{assumption}

Strategies to conform the search directions to the geometry of the nearby feasible region, i.e. to satisfy Assumption~\ref{condition:geometry_ass}, can be found in~\cite{Abramson_Brezhneva_DennisJr_Pingel_2008,Kolda_Lewis_Torczon_2007,Lewis_Torczon_2000,Lucidi_Sciandrone_Tseng_2002}.

\subsection{Preliminary results}

The proposed method is based on the sequential minimization of the merit function $Z(\cdot;\rho_k)$ for a decreasing sequence of penalty-barrier parameters $\set{\rho_k}$. In order for the method to work, an optimal solution has to exist for each Problem~\eqref{P2}. To guarantee the existence of an optimal solution for a minimization problem, one of the most standard assumptions is the compactness of the lower-level sets of the objective function. In our method, this would correspond to the compactness of the lower level sets $\Ll_\rho(\alpha)=\set{\vecx\in\X\mid Z(\vecx;\rho)\leq \alpha}$ for all $\rho>0$. Note that $Z(\vecx;\rho)$ is evaluated only at points lying in $\X$. To this end, the following assumption is considered, and the desired result is proved in Lemma~\ref{lemma_compactness}.


\begin{assumption}
\label{ass:assumption1}
The set $\X\cap\Omegaint$ is nonempty and compact.
\end{assumption}

\begin{lemma}
\label{lemma_compactness}
Let Assumption~\ref{ass:assumption1} hold. Then, for all $\rho > 0$, $\nu\in(1,2]$, and $\alpha\in\R$ the lower-level set
\[
\Ll_\rho(\alpha) = \{\vecx\in\X\mid \ Z(\vecx;\rho)\leq\alpha\}
\]
is compact.
\end{lemma}
\begin{proof}
By definition, $\Ll_\rho(\alpha)$ is a subset of $\X$, and $Z(\cdot;\rho)$ is coercive with respect to $\Omegaint$, i.e. $Z(\vecx;\rho)\to +\infty$ as $\vecx\to\partial\Omegaint$, thus $\Ll_\rho(\alpha)\subseteq \X\cap\Omegaint$. Since, by Assumption~\ref{ass:assumption1}, $\X\cap\Omegaint$ is compact, then $\Ll_\rho(\alpha)$ is bounded.

It remains to prove that $\Ll_\rho(\alpha)$ is closed. To this end, let us show that, for any sequence $\{\vecx_k\}\subset \Ll_\rho(\alpha)$ such that $\lim_{k\to+\infty}\vecx_k = \bar\vecx$, it results $\bar \vecx\in \Ll_\rho(\alpha)$.

Since $\vecx_k\in \Ll_\rho(\alpha)$ for all $k$, it follows
\begin{equation}
\label{g1}
Z(\vecx_k;\rho) \leq\alpha.
\end{equation}
The merit function $Z(.;\rho)$ is continuous at all
$\vecx\in \Ll_\rho(\alpha)$. Thus,
\[
\lim_{k\to+\infty}Z(\vecx_k;\rho) = Z(\bar\vecx;\rho)\leq\alpha,
\]
which means that $\bar\vecx\in \Ll_\rho(\alpha)$ and concludes the proof.
\end{proof}

Note that Assumption~\ref{ass:assumption1} arises from the observation that the logarithmic term of $Z(\cdot;\rho)$ might go to $-\infty$ if also $g_\ell(\cdot)\to -\infty$ for some $\ell\in\Gint$. Since the functions are assumed to be continuous on $\X$, the compactness assumption prevents such situation from happening. {Also}, since $\X$ is defined by the linear inequalities of the problem, then $\X$ also encompasses the possible bounds on the variables. In practical applications, one is usually able to identify some finite bounds wherein the optimal solution should lie. Therefore, the assumption that $\X$ is compact can be considered realistic. 

Lemma~\ref{lemma_compactness} is used to prove properties that form the basis of the theoretical convergence analysis of~\ref{LOG-DS}.  In particular, it is exploited in the proof of Theorem~\ref{convergence_eps}. Furthermore, it is important to highlight that the results presented in this work do not pertain to the entire sequence generated by the proposed algorithm, but focus on the subset of iterations wherein the penalty-barrier parameter $\rho_k$ is updated, i.e., the sequence of inexact solutions of the linearly constrained Problems~\eqref{P2}. Following the terminology used in the interior point methods literature (see~\cite{Bertsekas_1999}), the sequence of inexact solutions generated by~\ref{LOG-DS} is referred to as a {\textit{path-following}} sequence, and the corresponding index set is denoted by $\K_\rho$.
\begin{definition}
Let $\K_\rho$ be defined as:
\[
\K_\rho = \{k\in\N\mid\ \rho_{k+1}<\rho_k\}.
\]
Given the sequence of iterates $\{\vecx_k\}$ produced by  Algorithm~\ref{LOG-DS}, the subsequence $\set{\vecx_k}_{k\in \K_\rho}$ is said to be a path-following subsequence.
\end{definition}

{Recall that the set of unsuccessful iteration indexes is defined as $\K_u=\{k\in\N\mid\ \vecx_{k+1} = \vecx_k\}$, and the set of successful ones as $\K_s=\{k\in\N\mid\ \vecx_{k+1} \neq \vecx_k\}$}. Let us recall that \textbf{Step 3} is invoked only at unsuccessful iterations, i.e., $k\in \K_u$, so one has $\K_\rho\subseteq\K_u$. Additionally, according to the instructions of the algorithm, if $k\in\K_\rho$ then the following criterion is satisfied: 
\begin{equation}
\label{eq:penalty_criterion}
\alpha_{k+1}\leq \set{\rho_k^\beta, (\gmink)^2},
\end{equation}
and $\rho_{k+1}=\theta_\rho \rho_k$, with $\theta_\rho\in(0,1)$. From the updating procedure, it follows that $\set{\rho_k}$ goes to zero if and only if $\K_\rho$ is infinite.

Theorem~\ref{convergence_eps} shows that $\K_\rho$ is infinite and, consequently, that the sequence of penalty-barrier parameters converges to zero, which is required to ensure that, in the limit, the algorithm is able to solve the original problem. Furthermore, it is possible to establish that the corresponding sequence of step-sizes also converges to zero. As mentioned above, since the step-size is related to some measure of stationarity, this sets the ground to prove convergence to stationary points.

\begin{theorem}
\label{convergence_eps} Let Assumption~\ref{ass:assumption1} hold. Let $\set{\vecx_k}_{k\in\N}$, $\set{\rho_k}_{k\in \mathbb{N}}$, and $\set{\alpha_k}_{k\in \mathbb{N}}$ be the sequences of iterates, penalty parameters and step-sizes generated by LOG-DS, respectively. 
Then, $\K_\rho$ is infinite and
\begin{equation}\label{eps_to_0}
\Disp \lim_{k\to +\infty} \rho_k=0.
\end{equation}

Furthermore, it holds
\begin{equation}\label{alpha_to_0}
\Disp \lim_{k\to +\infty,k\in\K_\rho} \alpha_k=0.
\end{equation}
\end{theorem}

\begin{proof}
Let us start by proving that $\K_\rho$ is infinite. By contradiction, assume that $\K_\rho$ is finite, and without loss of generality, let us assume $\K_\rho=\emptyset$. Thus, $\rho_k = \rho_0>0$ for all $k$.

Let us consider the sets of successful and unsuccessful iterations, $\K_s$ and $\K_u$, respectively. By the instructions of the algorithm, it follows 
\begin{equation}\label{eq:K_s}
Z(\vecx_{k+1};\rho_0)\leq Z(\vecx_k;\rho_0) -\xi(\alpha_k), \quad \forall k\in \K_s,
\end{equation}
\begin{equation}\label{eq:K_u}
Z(\vecx_{k+1};\rho_0)=Z(\vecx_k;\rho_0), \quad \forall k\in \K_u.
\end{equation}
Hence, the sequence of function values $\{Z(\vecx_k;\rho_0)\}$ is monotonically non-increasing. By Lemma~\ref{lemma_compactness}, $Z(\cdot;\rho_0)$ has compact lower-level sets, thus it is bounded below. Hence,
\begin{equation}\label{utile4}
\lim_{k\to +\infty}Z(\vecx_k;\rho_0) = \bar Z.
\end{equation}

\noindent If $\K_s$ is infinite, then~\eqref{utile4} and~\eqref{eq:K_s} allow to conclude
\begin{equation*}
\lim_{k\to +\infty \atop k\in \K_s} \xi(\alpha_k) = 0.
\end{equation*}
Recalling Definition~\ref{def:forcing_function}, the following limit is obtained
\begin{equation}
\label{utile5}
\lim_{k\to +\infty \atop k\in \K_s} \alpha_k = 0.
\end{equation}
If $\K_u$ is infinite and $\K_s$ is not empty, for every $k\in
\K_u$, let us define $m_k$ to be the largest index such that
$m_k\in \K_s$ and $m_k < k$. Then,
\[
\alpha_k = \alpha_{m_k}\phi\theta_\alpha^{k-m_k-1}.
\]
When $k\to +\infty$, $k\in \K_u$, either $m_k\to +\infty$ as well (when $\K_s$ is
infinite) or~$k-m_k-1\to +\infty$ (when $\K_s$ is finite). Thus,
by~\eqref{utile5} and the fact that $\theta_\alpha\in(0,1)$, we
have
\begin{equation}\label{utile6}
\lim_{k\to +\infty\atop k\in \K_u}\alpha_k = \lim_{k\to +\infty \atop k\in \K_u}\alpha_{m_k}\phi\theta_{\alpha}^{k-m_k-1} = 0.
\end{equation}
If $\K_u$ is infinite and $\K_s$ is empty, then $\alpha_k =
\alpha_{0}\theta_\alpha^{k}$ and
\begin{equation}\label{utile_extra} \lim_{k\to +\infty\atop k\in
\K_u}\alpha_k=0, \end{equation} follows.

Considering~\eqref{utile5},~\eqref{utile6},
and~\eqref{utile_extra}, it can be concluded that
\begin{equation}\label{utile_extra2}\lim_{k\to+\infty} \alpha_k = 0.\end{equation}

By the instructions of the algorithm, for $k\in \K_u$ sufficiently
large, since $k\notin \K_\rho$ and
recalling~\eqref{eq:penalty_criterion}, it holds
\begin{equation*}
\alpha_{k+1} > \min\{\rho_0^{\beta},(\gmink)^2\}.
\end{equation*}
The latter, recalling~\eqref{utile_extra2}, implies
\begin{equation*}
\lim_{k\to+\infty,\atop k\in \K_u} \gmink = 0,
\end{equation*}
which, by the definition of $Z(\cdot;\rho_0)$, further implies
\begin{equation*}
\lim_{k\to+\infty,\atop k\in \K_u} Z(\vecx_k;\rho_0) = +\infty.
\end{equation*}
This is a contradiction with the monotonicity of
$\set{Z(\vecx_k;\rho_0)}$, thus proving that $\K_\rho$ is
infinite.\\

\noindent Let us now prove~\eqref{eps_to_0}. Let $\K_\rho=\set{k_1,k_2,\dots,k_j,\dots}$. By the instructions of the algorithm,
\begin{equation*}
\rho_{k_j} = \theta_\rho \rho_{k_{j-1}} = \theta_{\rho}^j \rho_0.
\end{equation*}
Since $\K_\rho$ is infinite, the limit for
$j\to+\infty$ can be taken, and, recalling $\theta_\rho\in(0,1)$,
\begin{equation*}
\lim_{j\to+\infty} \rho_{k_j} = \lim_{j\to+\infty} \theta_{\rho}^j
\rho_0 = 0,
\end{equation*}
which, noting $\rho_{k+1}\leq \rho_k$ for all $k$, proves~\eqref{eps_to_0}.\\

\noindent {Finally,~\eqref{alpha_to_0} will be established.}
By definition of $\K_\rho$, for all $k\in\K_\rho$
\[
\alpha_{k+1} \leq \min\{\rho_k^{\beta},(\gmink)^2\}\leq \rho_k^\beta.
\]
Furthermore, for all iterations $\alpha_{k+1}\geq \theta_\alpha\alpha_k$. So, for all $k\in\K_\rho$:
\[
\theta_\alpha\alpha_k \leq \alpha_{k+1} \leq \rho_k^\beta.
\]
Then, the proof follows by recalling that
$\lim_{k\to+\infty}\rho_k = 0$.
\end{proof}

\section{Convergence analysis}
\label{convergence_analysis}
{This section presents a theoretical analysis of the convergence properties of Algorithm~\ref{LOG-DS}. Its main result (Theorem~\ref{main_theorem}) relies on optimality
conditions derived using the following extended Mangasarian-Fromovitz
constraint qualification (MFCQ) {(see~\cite[Sec. 3.3]{Bertsekas_1999})}}.
\begin{definition}
\label{MFCQ} The point $\vecx\in \X$ is said to satisfy the MFCQ for Problem~\eqref{P} if the two following conditions are satisfied:
\begin{itemize}
\item[\rm (a)] There does not exist a nonzero vector
$\alpha=(\alpha_1,\ldots,\alpha_p)$ such that:
\begin{equation}
\label{box_mfcq} \left(\sum_{i=1}^{p} \alpha_i\nabla
h_i(\vecx)\right)^\top \vecd \geq 0,\quad\quad\forall \vecd\in
\T_\X(\vecx).
\end{equation}
\item[\rm (b)] There exists a feasible direction $\vecd\in \T_\X(\vecx)$, such that:
\begin{equation}
\label{mfcq_feas}
\nabla g_\ell(\vecx)^\top \vecd < 0,\ \ \ \forall \ell\in \I_+(\vecx),\quad \nabla h_j(\vecx)^\top \vecd=0,\ \ \ \forall j=1,\ldots,p
\end{equation}
where $\I_+(\vecx)=\{\ell\in\{1,\ldots,m\} \mid
g_\ell(\vecx)\geq 0\}$.
\end{itemize}
\end{definition}

Consider the Lagrangian function $L(\vecx,\lambda,\bvec{\mu})$ associated with the nonlinear constraints of Problem~\eqref{P}, defined by: $$L(\vecx,\bvec{\lambda},\bvec{\mu})= f(\vecx)+\bvec{\lambda}^\top g(\vecx)+\bvec{\mu}^\top h(\vecx),$$
{where $g(\vecx)=[g_1(\vecx)\ldots g_m(\vecx)]^{\top}$ and $h(\vecx)=[h_1(\vecx)\ldots h_p(\vecx)]^{\top}$.}
The following proposition is a well-known result (see~\cite[Sec. 3.3]{Bertsekas_1999}), which states necessary optimality conditions for Problem~\eqref{P}.

\begin{proposition}
\label{bert}
Let $\vecx^*\in\cal F$ be a local minimum of Problem~{\rm \eqref{P}} that satisfies the MFCQ. Then, there exist vectors $\lambda^*\in \R^m$, $\mu^*\in\R^p$ such that
\begin{equation}
\label{derdir} \nabla_x L(\vecx^*,\lambda^*,\mu^*)^\top (\vecx-\vecx^*)\geq 0,\quad\quad\forall \vecx\in \X
\end{equation}
\begin{equation}
\label{cmp} (\lambda^*)^{\top} g(\vecx^*)=0,\quad \lambda^*\geq 0.
\end{equation}
\end{proposition}

Therefore, the following definition of stationarity is considered.
\begin{definition}[Stationary point]
A point $\vecx^*\in\F$ is said to be a stationary point for Problem~\eqref{P} if vectors $\lambda^*\in\R^m$ and $\mu^*\in\R^p$ exist such that \eqref{derdir} and \eqref{cmp} are satisfied.
\end{definition}

Before presenting the main convergence result, a technical proposition is established, which will be invoked in Theorem~\ref{main_theorem}. The need for this proposition stems from noting that the trial points corresponding to failures might not strictly satisfy the inequality constraints with indices $\ell\in\Gint$, implying that the merit function value is, by definition, $+\infty$.

\begin{proposition}
\label{feasibility} Given any $\bar{\rho}>0$, $\vecx\in\X$ such that $g_\ell(\vecx)<0$ for all $\ell\in\Gint$, $\vecd\in \R^n$, and $\bar{\alpha}\in\R_+$ such that $\vecx+\bar{\alpha} \vecd\in \X$ and
\[
Z(\vecx+\bar\alpha \vecd;\bar{\rho}) > Z(\vecx;\bar{\rho})-\xi(\bar{\alpha}),
\]
there exists $\hat\alpha\leq\bar\alpha$ such that:
\[
\begin{split}
&\vecx+\alpha \vecd\in\X\quad \forall\alpha\in(0,\hat\alpha],\\
&\max_{\ell\in\Gint}g_\ell(\vecx+\alpha \vecd)<0\quad \forall\ \alpha\in(0,\hat\alpha],\\
& Z(\vecx+\hat\alpha \vecd;\bar{\rho}) > Z(\vecx;\bar{\rho})-\xi(\hat{\alpha}).
\end{split}
\]
\end{proposition}
\begin{proof}
By definition $\X$ is convex, then $\vecx+\alpha\vecd\in\X$ for all $\alpha \in [0,\bar{\alpha}]$. Thus, if $\max_{\ell\in\Gint}g_\ell(\vecx+\alpha \vecd)<0$ for all $\alpha\in(0,\bar\alpha]$, by setting $\hat\alpha=\bar\alpha$, the proposition holds.\smallskip
\\\noindent
Let us assume $\max_{\ell\in\Gint}g_\ell(\vecx+\alpha \vecd)\geq 0$ for some $\alpha\in(0,\bar\alpha]$. Since $g_\ell(\vecx)<0$ for all $\ell\in\Gint$, it follows that $\max_{\ell\in\Gint}g_\ell(\vecx)<0$. Hence, by continuity of $g_\ell(\cdot)$, there exists a scalar $\tilde{\alpha}<\bar{\alpha}$ such that
\[
\max_{\ell\in\Gint}g_\ell(\vecx+\alpha\vecd)<0\ \text{ for all }\alpha\in[0,\tilde{\alpha}),
\]
\[
\max_{\ell\in\Gint}g_\ell(\vecx+\tilde{\alpha}\vecd)=0.
\]
Recall that, by definition of $Z(\cdot;\bar{\rho})$, it
holds that
\[
\lim_{\alpha\to\tilde\alpha}Z(\vecx+\alpha\vecd;\bar{\rho}) = +\infty.
\]
Thus, there exists $\hat\alpha \in (0,\tilde\alpha)$, sufficiently close to $\tilde\alpha$, such that  $g_\ell(\vecx+\hat\alpha \vecd)<0$, for all $\ell\in\Gint$, and $Z(\vecx;\bar{\rho})< Z(\vecx+\hat{\alpha}\vecd;\bar{\rho})+\xi(\hat{\alpha})$, which concludes the proof.
\end{proof}\medskip
\medskip

\noindent The main convergence result is stated next. 
\begin{theorem}
\label{main_theorem}
Let Assumption~\ref{ass:assumption1} hold, $\set{\vecx_k}_{k\in \mathbb{N}}$ be the sequence of iterates generated by~\ref{LOG-DS}, and $\K_\rho$ be a {path-following} sequence. Assume that the sets of directions $\{\Dd_k\}_{k\in \mathbb{N}}$, used by the algorithm, satisfy Assumption~\ref{condition:geometry_ass} and define $\J_k = \{i\in\{1,2,\dots,|\Dd_k|\}\mid\ \vecd_k^i\in \Dd_k\cap \T_\X(\vecx_k;\varepsilon)\}$, with $\varepsilon\in(0,\min\{\bar\varepsilon,\varepsilon^*\}]$ where $\varepsilon^*$ and $\bar\varepsilon$ are the  constants appearing in Proposition \ref{prop:epsaccurate_estimate} and Assumption~\ref{condition:geometry_ass}, respectively. Then, any limit point of $\{\vecx_k\}_{k\in \K_\rho}$ that satisfies the MFCQ is a stationary point of Problem~\eqref{P}.
\end{theorem}

\begin{proof}
First note that, by Theorem \ref{convergence_eps}, 
$$\lim_{k\to+\infty\atop k\in \K_\rho}\rho_k = 0,\quad \text{ and }\ \lim_{k\to+\infty\atop k\in \K_\rho}\alpha_k = 0.$$

Now, let $\vecx^*$ be any limit point of $\{\vecx_k\}_{k\in \K_\rho}$. Then, there exists a set $\K_\rho^{\tt x} \subseteq \K_\rho$ such that
$$\lim_{k\to+\infty\atop k\in \K_\rho^{\tt x}}\rho_k = 0,\quad\ \lim_{k\to+\infty\atop k\in \K_\rho^{\tt x}}\alpha_k = 0,\quad \text{ and }\ \lim_{k\to+\infty\atop k\in \K_\rho^{\tt x}}\vecx_k = \vecx^*, $$
with $\alpha_{k+1}<\alpha_k$ and $\rho_{k+1}<\rho_k$, for all $k\in\K_\rho^{\tt x}$.
Recall that $\Dd_k = \{\vecd_k^1,\vecd_k^2,\dots,\vecd_k^{|\Dd_k|}\}$.
Then, for all $k\in\K_\rho^{\tt x}$ sufficiently large, $\vecx_k+\alpha_k\vecd_k^i\in \X$ for all $i\in \J_k$. For every $i\in \J_k$, if $g_\ell(\vecx_k+\alpha_k\vecd_k^i)<0$, for all $\ell\in\Gint$, then by the instructions of the algorithm
\begin{equation}\label{uns_iter_1}
Z(\vecx_k+{\alpha}_k \vecd_k^i;\rho_k)>
Z(\vecx_k;\rho_k)-\xi({\alpha}_k).
\end{equation}
Otherwise, i.e. when an index $\ell\in\Gint$ exists such that $g_\ell(\vecx_k+\alpha_k\vecd_k^i)\geq 0$, i.e. $Z(\vecx^k+\alpha_k\vecd_k^i;\rho_k) = +\infty$, 
Proposition~\ref{feasibility}
allows us to ensure the existence of a scalar
$\hat{\alpha}_k^i\leq \alpha_k$ such that
\begin{equation}
\label{uns_iter} Z(\vecx_k+\hat{\alpha}_k^i \vecd_k^i;\rho_k)>
Z(\vecx_k;\rho_k)-\xi(\hat{\alpha}_k^i).
\end{equation}
Applying the mean value theorem to~\eqref{uns_iter} (or \eqref{uns_iter_1} setting $\hat\alpha_k^i=\alpha_k$), it follows
\begin{equation}
-\xi(\hat{\alpha}_k^i)\leq Z(\vecx_k+\hat{\alpha}_k^i \vecd_k^i;\rho_k)
- Z(\vecx_k;\rho_k) = \hat{\alpha}_k^i \nabla Z(\vecy_k^i;\rho_k)^\top
\vecd_k^i,
\end{equation}
for all $k\in\K_\rho^{\tt x}$ sufficiently large and all $i\in \J_k$, where
$\vecy_k^i=\vecx_k+t_k^i\hat{\alpha}_k^i \vecd_k^i$, with $t_k^i\in
(0,1)$. Thus,
\begin{equation}
\nabla Z(\vecy_k^i;\rho_k)^\top \vecd_k^i \geq
-\frac{\xi(\hat{\alpha}_k^i) }{\hat{\alpha}_k^i}, \quad
\forall i\in \J_k.
\end{equation}
By considering the expression of $Z(\vecx;\rho_k)$, for all $i$ in $\J_k$ and for $k\in\K_\rho^{\tt x}$ sufficiently large
\begin{align}
\nabla Z(\vecy_k^i;\rho_k)^\top \vecd_k^i &= \left[\nabla f(\vecy_k^i) +
\sum_{\ell\in \Gint} \dfrac{\rho_k}{-g_\ell(\vecy_k^i)}\nabla
g_\ell(\vecy_k^i) +
\nu \left(\sum_{\ell\in \Gext}\left(\frac{\max\{g_\ell(\vecy_k^i),0\}}{\rho_k}\right)^{\nu-1}\nabla g_\ell(\vecy_k^i) +\right.\right. \nonumber \\ \label{eq:grad_Z}
& \left. \left. \hspace{-0.8cm}\sum_{j=1}^{p} \left(\frac{|h_j (\vecy_k^i)|}{\rho_k}\right)^{\nu-1} \nabla h_j (\vecy_k^i)\right)\right]^{\top}\vecd_k^i \geq-\frac{\xi(\hat{\alpha}_k^i)}{\hat{\alpha}_k^i}.
\end{align}
Proposition~\ref{prop:epsaccurate_estimate} allows to extract
a further subset of indices $\K_\rho^{\tt
x,D}\subseteq\K_\rho^{\tt x}$ such that the sets of poll directions $\Dd_k$ are equal to a fixed set $\Dd^*$ for all $k\in\K_\rho^{\tt x,D}$. Therefore,
\[
\begin{split}
& \lim_{k\to+\infty\atop k\in \K_\rho^{\tt x,D}}\rho_k = 0,\\
& \lim_{k\to+\infty\atop k\in \K_\rho^{\tt x,D}}\alpha_k = 0,\\
& \lim_{k\to+\infty\atop k\in \K_\rho^{\tt x,D}}\vecx_k = \vecx^*,\\
& \J_k = \J^*,\quad\forall\ k\in\K_\rho^{\tt x,D},\\
& \vecd_k^i = \bar \vecd^i,\quad\forall\ i\in \J^*,\ k\in\K_\rho^{\tt x,D},\\
& \Dd^*=\{\bar \vecd^i\}_{i\in \J^*}.
\end{split}
\]
When $k\in\K_\rho^{\tt x,D}$ is sufficiently large, for all $i\in \J^*$, with $\vecy_k^i=\vecx_k+t_k^i\hat{\alpha}_k^i\bar\vecd^i$, and $t_k^i\in (0,1)$, since $\hat{\alpha}_k^i\leq \alpha_k$, by Theorem~\ref{convergence_eps}, we have that,  $\Disp\lim_{k\to+\infty \atop k\in\K_\rho^{\tt x,D}} \vecy_k^i = \vecx^*$.

Let us define the following approximations to the Lagrange multipliers of each constraint:
\begin{itemize}
\item[-] for $\ell=1,\ldots,m$, set
$\lambda_\ell(\vecx;\rho) = \begin{cases} \dfrac{\rho}{-g_\ell(\vecx)}, & \text { if } \ell\in\Gint\\ \nu \left( \dfrac{\max\{g_\ell(\vecx),0\}}{\rho}\right)^{\nu-1}, & \text { if } \ell\in\Gext \end{cases}$
\item[-] for $j=1,\ldots,p$, set $\mu_j(\vecx;\rho) = \nu\left(\dfrac{ |h_j(\vecx)|}{\rho}\right)^{\nu-1}$.
\end{itemize}


The sequences $\set{\lambda_\ell(\vecx_k;\rho_k)}_{k\in \K_\rho^{\tt x,D}}$ and $\set{\mu_j(\vecx_k;\rho_k)}_{k\in \K_\rho^{\tt x,D}}$, are bounded (see Appendix \ref{appendix}).
Thus, it is possible to consider $\K_\rho^{\tt
x,D,\lambda}\subseteq\K_\rho^{\tt x,D}$, such that

\begin{eqnarray}
&&\Disp\lim_{k\to+\infty \atop k\in \K_\rho^{\tt x,D,\lambda}}\lambda_\ell(\vecx_k;\rho_k) = \lambda_\ell^*, \qquad \ell=1,\ldots,m\\
&&\Disp\lim_{k\to+\infty \atop k\in \K_\rho^{\tt x,D,\lambda}} \mu_j(\vecx_k;\rho_k) = \mu_j^*, \qquad j=1,\ldots,p.
\end{eqnarray}
Note that, by definition of the multiplier functions, $\lambda_\ell^*=0$ for $\ell\not\in \I_+(\vecx^*)$.

Let us now multiply~\eqref{eq:grad_Z} by $\rho_k^{\nu-1}$ and take the limit for $k\to+\infty, k\in\K_\rho^{\tt x,D,\lambda}$ . Recall that $\nu\in(1,2]$, so $\rho_k^{\nu-1}\to 0$ and
\[
\left(\sum_{\ell\in \Gext}\nu\max\{g_\ell(\vecx^*),0\}^{\nu-1}\nabla g_\ell(\vecx^*)+\sum_{j=1}^{p}\nu |h_j (\vecx^*)|^{\nu-1} \nabla h_j
(\vecx^*)\right)^\top {\bar{\vecd}^i}\geq 0, \quad \forall \bar\vecd^i\in D^*.
\]

From Proposition~\ref{prop:epsaccurate_estimate} and Assumption~\ref{condition:geometry_ass}, since $\varepsilon\in(0,\min\set{\bar{\varepsilon},\varepsilon^*}$], for all $k\in\K_\rho^{\tt x,D,\lambda}$ sufficiently large it holds
\[
\T_\X(\vecx^*) = \T_\X(\vecx_k;\varepsilon) = \cone(\Dd_k\cap \T_\X(\vecx_k;\varepsilon))= \cone(\Dd^*).
\]
Then, for every $\vecd\in \T_\X(\vecx^*)$, there exist nonnegative numbers $\beta_i$ such that
\begin{equation}
\label{direction_comb}
\vecd = \sum_{i\in \J^*} \beta_i \bar\vecd^{i}, \text{with } \bar\vecd^{i}\in \Dd^*.
\end{equation}
Let us recall that, by assumption, $\vecx^*$ satisfies the MFCQ conditions, and let $\vecd$ be the direction satisfying \eqref{mfcq_feas} in point (b) of the MFCQ conditions. Then,
\begin{eqnarray}
\nonumber 0\leq && \sum_{i\in \J^*}\beta_i \left(\sum_{\ell\in \Gext}\nu\max\{g_\ell(\vecx^*),0\}^{\nu-1}\nabla g_\ell(\vecx^*) +\sum_{j=1}^{p}\nu |h_j (\vecx^*)|^{\nu-1} \nabla h_j (\vecx^*)\right)^\top {\bar{\vecd}^i} =
\\
\label{eq:derdir_positive}&& \left(\sum_{\ell\in \Gext}\nu\max\{g_\ell(\vecx^*),0\}^{\nu-1}\nabla g_\ell(\vecx^*)+\sum_{j=1}^{p}\nu |h_j (\vecx^*)|^{\nu-1} \nabla h_j (\vecx^*)\right)^\top \vecd =\\
\nonumber && \left(\sum_{\ell\in \I_+(\vecx^*)\cap \Gext}\nu\max\{g_\ell(\vecx^*),0\}^{\nu-1}\nabla g_\ell(\vecx^*)+\sum_{j=1}^{p}\nu |h_j (\vecx^*)|^{\nu-1} \nabla h_j (\vecx^*)\right)^\top \vecd,
\end{eqnarray}
where the last equality follows by considering that
$\max\{g_\ell(\vecx^*),0\} = 0$, for all
$\ell\in\Gext\setminus\I_+(\vecx^*)$.

Again by (\ref{mfcq_feas}), $\nabla g_\ell(\vecx^*)^T \vecd < 0,$ for all $\ell\in \I_+(\vecx^*)$, and $\nabla h_j(\vecx^*)^\top \vecd = 0$, for all $j$. Then, $\max\{g_\ell(\vecx^*),0\}=0$ for all $\ell\in \I_+(\vecx^*)\cap \Gext$, implying
\begin{eqnarray*}
&& \hspace{4cm}\left(\sum_{j=1}^{p}\nu |h_j (\vecx^*)|^{\nu-1} \nabla h_j (\vecx^*)\right)^\top {{\vecd}} \geq 0,\ \ \text{for all } \vecd\in \T_\X(\vecx^*).
\end{eqnarray*}

By (\ref{box_mfcq}), $h_j(\vecx^*) = 0$ for all $j=1,\dots,p$. Therefore, recalling that $g_\ell(\vecx^*)\leq 0$ for all $\ell\in\Gint$, we obtain that point $\vecx^*$ is feasible.

By simple manipulations, inequality~\eqref{eq:grad_Z} can be
rewritten as
\begin{equation}
\begin{aligned}
& \left(\nabla f(\vecy_k^i)+\sum_{\ell=1}^m \nabla g_\ell(\vecy_k^i) \lambda_\ell(\vecx_k;\rho_k)\right. \\
& +\sum_{\ell=1}^m \nabla g_\ell(\vecy_k^i) (\lambda_\ell(\vecy_k^i;\rho_k)-\lambda_\ell(\vecx_k;\rho_k))+\sum_{j=1}^p \nabla h_j (\vecy_k^i) \mu_j(\vecx_k;\rho_k) \\
& \left.+\sum_{j=1}^p \nabla
h_j(\vecy_k^i)(\mu_j(\vecy_k^i;\rho_k)-\mu_j(\vecx_k;\rho_k))\right)^\top
\bar\vecd^i \geq-\frac{\xi(\hat{\alpha}_k^i)}{\hat{\alpha}_k^i}, \quad
\forall i \in \J^*
\end{aligned}
\end{equation}
Taking limits for $k\to +\infty$, $k\in\K_\rho^{\tt x,D,\lambda}$ and
considering~\eqref{eq_ap:lim_lambda} and~\eqref{eq_ap:lim_mu} in Appendix~\ref{appendix}, it follows:
\begin{eqnarray}
\left(\nabla f(\vecx^*) + \sum_{\ell=1}^{m} \nabla g_\ell(\vecx^*)\lambda_\ell^* +
\sum_{j=1}^{p}  \nabla h_j (\vecx^*)\mu_j^*\right)^\top{\bar\vecd^i} \geq 0,
\quad \forall i\in \J^*.\label{derdirge0}
\end{eqnarray}
Again, by \eqref{direction_comb} and \eqref{derdirge0}, for all $\vecd\in \T_\X(\vecx^*)$,
\begin{eqnarray*}
&& \left(\nabla f(\vecx^*) + \sum_{\ell=1}^{m} \nabla g_\ell(\vecx^*)\lambda_\ell^* + \sum_{j=1}^{p} \nabla h_j (\vecx^*) \mu_j^*\right)^\top \vecd = \\
&&\qquad \sum_{i\in \J^*}\beta_i\left(\nabla f(\vecx^*) +
\sum_{\ell=1}^{m} \nabla g_\ell(\vecx^*)\lambda_\ell^* + \sum_{j=1}^{p} \nabla
h_j (\vecx^*)\mu_j^*\right)^\top{ \bar\vecd^{i}} \geq 0.
\end{eqnarray*}
Since $\vecx^*$ is feasible and considering that, by definition of $\lambda_\ell^*$ for {$\ell\notin \I_+(\vecx^*)$}, $\lambda_\ell^*\, g_\ell (\vecx^*)=0$, for all $\ell=1,\ldots,m$, $\vecx^*$ is a KKT point for problem \eqref{P}, thus concluding the proof.
\end{proof}

\section{Implementation details of LOG-DS}\label{sec:implementation}
This section presents a practical implementation for LOG-DS, based on the original implementation of SID-PSM~\cite{Custodio_Rocha_Vicente_2010,Custodio_Vicente_2007}.

\subsection{LOG-DS vs SID-PSM}
\label{subsec:implementation_details}

LOG-DS enhances SID-PSM with the ability of handling general constraints using a penalty-interior point method. Thus, the original structure and algorithmic options of SID-PSM implementation were considered. This subsection provides a general overview of the main features of SID-PSM, highlighting the differences with LOG-DS. 

SID-PSM{~\cite{Custodio_Rocha_Vicente_2010,Custodio_Vicente_2007}} is a GPS method that leverages quadratic polynomial models and simplex gradients to improve the numerical performance. Each iteration of SID-PSM consists in an optional Search Step followed by a Poll Step. 

During the Search Step, SID-PSM constructs quadratic polynomial models around the current iterate using previously evaluated points, sampled within a ball whose radius is proportional to the step-size parameter. A minimum threshold of points is required to build the models. Thus, the Search Step is skipped in the first iterations. Depending on the number of available points, the model may be underdetermined, determined, or overdetermined. The model is minimized within a trust region, and the resulting candidate, if feasible, is evaluated in the true objective function; otherwise, its value is set to infinity. If a new feasible point with a better objective value is found, the Poll Step is skipped and the process repeats from the new point. Otherwise, the algorithm proceeds to the Poll Step.


In the Poll Step, a local search is performed by evaluating a set of positive spanning directions, scaled by the step-size parameter, in an opportunistic manner. Therefore, the testing order can affect performance, as not all directions need to be evaluated. Previously sampled points are used to compute a \textit{simplex gradient}~\cite{Bortz_Kelley_1998}, which serves as an ascent indicator to reorder the poll directions. In a simplified way, a simplex gradient can be regarded as the gradient of a linear interpolation model. Poll directions are reordered according to the largest angle made with this ascent indicator (see~\cite{Custodio_Vicente_2007}). The evaluation of the reordered poll points is made only on the feasible ones, setting the function value to infinity in case of infeasibility.
{If neither the Search Step nor the Poll Step yield a feasible point with a better objective value, SID-PSM reduces the step-size parameter.}



The main difference between LOG-DS and SID-PSM is the use of a merit function to address constraints, instead of an extreme barrier approach. In SID-PSM, only feasible points are evaluated, being the function value set equal to $+\infty$ for infeasible ones. LOG-DS allows the violation of some constraints. The merit function also replaces the original objective function throughout the different algorithmic steps. In the Search Step, quadratic polynomial models are built for both the objective and constraints functions and they are analytically combined to produce a model of the merit function.

The sets of points used in model computation do not require feasibility regarding the nonlinear constraints, always resulting from previous evaluations of the merit function. No function evaluations are spent solely for the purpose of model building. Once more, depending on the number of points available, minimum Frobenius norm models, quadratic interpolation, or regression approaches can be used~\cite{Custodio_Rocha_Vicente_2010} to compute the model coefficients.

The original SID-PSM only requires a simple decrease for the
acceptance of new points, whereas in LOG-DS points are accepted if
they satisfy the sufficient decrease condition
$$Z(\vecx_{k+1};\rho_k) \leq Z(\vecx_k;\rho_k)-\gamma\alpha_k^2,$$
conferring flexibility to the updates of the step-size, since points are no longer constrained to an implicit mesh. The algorithm proceeds with an opportunistic Poll Step, accepting the first tentative point satisfying the sufficient decrease condition. Before initiating the polling procedure, previously evaluated points are again used to build a simplex gradient, which is used as an ascent indicator. In LOG-DS, the simplex gradient is built on the merit function values, while the original SID-PSM exploits the values of the objective function.

\subsection{Implementation choices}
A numerical implementation is always dependent on algorithmic choices and parameters. This subsection details the options taken for the computational implementation of LOG-DS,  with the exception of the poll directions, that are detailed in Subsection~\ref{sec:linear_comparison}.

The initialization requires setting the initial step-size, $\alpha_0 = 1$, the initial penalty-barrier parameter, $\rho_0 = 0.1$, the smoothing exponent for the exterior penalty, $\nu = 2$, the penalty-barrier update parameter, $\theta_{\rho} = 10^{-2}$, and the exponent used in the penalty-barrier parameter updating criterion $\beta = 1 + 10^{-9}$. In the sufficient decrease condition, $\gamma$ is set equal to $10^{-3}$.

Regarding the update of the step-size, the contraction  parameter is defined as $\theta_\alpha = 0.5$. 
{The step-size expansion parameter $\phi$ is set to $1$, so that the step-size is
unchanged on successful Poll Steps. A more elaborate update rule is however used in the Search Step. Recall that, in the Search Step of iteration
$k$,}
a surrogate model of $Z(\cdot,\rho_k)$ is minimized within a trust region of radius $r_k = 2\alpha_k$, centered at $\vecx_k$. Let $\vecs_k$ denote the corresponding minimizer. If 
\[ Z(\vecs_k;\rho_k)\leq Z(\vecx_k;\rho_k) - \gamma \alpha_k^2 \quad \text{ and } \quad \|\vecs_k - \vecx_k\| > \alpha_k,\]
then we set $$\alpha_{k+1} = \dfrac{1}{\sqrt[n]{\theta_{\alpha}}}\alpha_k,$$ where $n$ is the problem dimension. This scaling reflects the observation that reducing the step-size typically requires $\mathcal{O}(n)$ function evaluations. Since the step-size reduction is necessary for updating the penalty-barrier parameter and improving the solution accuracy, expanding the step-size incurs a hidden cost of function evaluations that grows with problem dimension. The proposed update heuristic procedure mitigates this cost, though it may slow down progress when the iterate is far from the solution. Notably, expansion occurs only when the successful trial point is far enough from the current iterate (i.e. when $\|\vecs_k - \vecx_k\|>\alpha_k$), preventing unnecessary updates during the final optimization phase, where surrogate models are well-known to be highly effective.

At any unsuccessful iteration $k$, the update test is relaxed by scaling both terms with positive constants. The penalty–barrier parameter $\rho_k$ is reduced whenever
$$\alpha_{k+1} \leq \min\{c_{\rho} \rho_{k}^{\beta}, c_g (\gmink)^2\},$$
where $c_\rho = 10^{2}$ and $c_g = 10^{10}$. For $\gmink$ on the order of $10^{-5}$, one has $c_g (\gmink)^2 \approx 1$, so the influence of proximity to $\Omegaint$ is retained, rather than masked by the scaling. Although one might expect such proximity to be attainable only when $\alpha_k$ is small, in practice the quadratic surrogate models used in the Search Step enable this precision even for larger step-sizes.

To balance the contributions within the merit function, the exterior penalty term is scaled by a constant $c_{\tt ext}$:
$$Z(\vecx;\rho_k) = f(\vecx) - \rho_k \sum_{\ell\in\Gint} \log(-g_\ell(\vecx)) + \dfrac{c_{\tt ext}}{\rho_k}\left(\sum_{\ell \in \Gext} \max\{g_\ell(\vecx),0\}^{\nu} + \sum_{j=1}^{p} |h_j(\vecx)|^{\nu}\right).$$
The constant $c_{\tt ext}$ is problem dependent and it is defined as:
$$ c_{\tt ext} = \begin{cases}
1 &\text{if } f(\vecx_0)=0,\\ \max\{1,10^{\lfloor\log_{10}(|f(\vecx_0|)\rfloor}\} & \text{otherwise},
\end{cases}$$
which matches the order of magnitude of the objective function at the initial point $\vecx_0$, so that constraint violations are given at least a comparable weight and the algorithm is driven toward feasibility as $\rho_k$ decreases.

Note that the three constants $c_\rho$, $c_g$, and $c_{\tt ext}$ do not affect the theoretical properties of the algorithm.

\section{Numerical experiments}
\label{numerical_experiments}
This section is dedicated to the numerical experiments and performance evaluation of the proposed mixed penalty-logarithmic barrier derivative-free optimization algorithm, LOG-DS, on a collection of test problems available in the literature.

The test set used in~\cite{Brilli_Liuzzi_Lucidi_2021}, part of the CUTEst collection~\cite{Gould_Orban_Toint_2015}, was considered. Specifically, problems were included for which the number of variables does not exceed $50$ and at least one inequality constraint is strictly satisfied by the initialization provided, ensuring $\vecx_0\in \intOmegaint$. Table~\ref{test_set} details the test set, providing the name, the number of variables, $n$, the number of inequality constraints, $m$, the number of inequality constraints treated by the logarithmic barrier $\bar{m}=|\Gint|$, and the number of nonlinear equalities, $p$, for each problem. Problems without equality constraints, for which the provided initialization is feasible, are marked in bold. Additional information can be found in~\cite{Brilli_Liuzzi_Lucidi_2021,Gould_Orban_Toint_2015}.

{\renewcommand{\arraystretch}{1.25}
\begin{table}[htp!]
\begin{center}\scriptsize
\begin{minipage}{0.48\textwidth}
\begin{center}
\begin{tabular}{|l|c|c|c|c|}\hline
Problem &  $n$ &  $m$ & $\bar{m}$ & $p$ \\\hline
ANTWERP & 27 &  10 &   2 &   8 \\
DEMBO7 & 16 &  21 &  16 &   0 \\
ERRINBAR & 18 &   9 &   1 &   8 \\
\textbf{HS117} & 15 &   5 &   5 &   0 \\
\textbf{HS118} & 15 &  29 &  28 &   0 \\
LAUNCH & 25 &  29 &  20 &   9 \\
LOADBAL & 31 &  31 &  20 &  11 \\
MAKELA4 & 21 &  40 &  20 &   0 \\
MESH & 33 &  48 &  17 &  24 \\
OPTPRLOC & 30 &  30 &  28 &   0 \\
RES & 20 &  14 &   2 &  12 \\
SYNTHES2 & 11 &  15 &   1 &   1 \\
SYNTHES3 & 17 &  23 &   1 &   2 \\
TENBARS1 & 18 &   9 &   1 &   8 \\
TENBARS4 & 18 &   9 &   1 &   8 \\
TRUSPYR1 & 11 &   4 &   1 &   3 \\
TRUSPYR2 & 11 &  11 &   8 &   3 \\
\textbf{HS12} &  2 &   1 &   1 &   0 \\
\textbf{HS13} &  2 &   1 &   1 &   0 \\
\textbf{HS16} &  2 &   2 &   2 &   0 \\
HS19 &  2 &   2 &   1 &   0 \\
\textbf{HS20} &  2 &   3 &   3 &   0 \\
\textbf{HS21} &  2 &   1 &   1 &   0 \\
HS23 &  2 &   5 &   4 &   0 \\
\textbf{HS30} &  3 &   1 &   1 &   0 \\
\textbf{HS43} &  4 &   3 &   3 &   0 \\
\textbf{HS65} &  3 &   1 &   1 &   0 \\
HS74 &  4 &   5 &   2 &   3 \\
HS75 &  4 &   5 &   2 &   3 \\
HS83 &  5 &   6 &   5 &   0 \\
HS95 &  6 &   4 &   3 &   0 \\
HS96 &  6 &   4 &   3 &   0 \\
HS97 &  6 &   4 &   2 &   0 \\
HS98 &  6 &   4 &   2 &   0 \\
\textbf{HS100} &  7 &   4 &   4 &   0 \\
HS101 &  7 &   6 &   2 &   0 \\
HS104 &  8 &   6 &   3 &   0 \\
\textbf{HS105} &  8 &   1 &   1 &   0 \\
\textbf{HS113} & 10 &   8 &   8 &   0 \\
HS114 & 10 &  11 &   8 &   3 \\
HS116 & 13 &  15 &  10 &   0 \\
S365 &  7 &   5 &   2 &   0 \\
ALLINQP & 24 &  18 &   9 &   3 \\
BLOCKQP1 & 35 &  16 &   1 &  15 \\
BLOCKQP2 & 35 &  16 &   1 &  15 \\
BLOCKQP3 & 35 &  16 &   1 &  15 \\
BLOCKQP4 & 35 &  16 &   1 &  15 \\
BLOCKQP5 & 35 &  16 &   1 &  15 \\\hline
\end{tabular}
\end{center}
\end{minipage}
\begin{minipage}{0.49\textwidth}
\begin{center}
\begin{tabular}{|l|c|c|c|c|}\hline
Problem &  $n$ & $m$ & $\bar{m}$ & $p$\\\hline
CAMSHAPE & 30 &  94 &  90 &   0 \\
CAR2 & 21 &  21 &   5 &  16 \\
CHARDIS1 & 28 &  14 &  13 &   0 \\
EG3 & 31 &  90 &  60 &   1 \\
GAUSSELM & 29 &  36 &  11 &  14 \\
\textbf{GPP} & 30 &  58 &  58 &   0 \\
HADAMARD & 37 &  93 &  36 &  21 \\
HANGING & 15 &  12 &   8 &   0 \\
JANNSON3 & 30 &   3 &   2 &   1 \\
\textbf{JANNSON4} & 30 &   2 &   2 &   0 \\
KISSING & 37 &  78 &  32 &  12 \\
KISSING1 & 33 & 144 & 113 &   0 \\
KISSING2 & 33 & 144 & 113 &   0 \\
LIPPERT1 & 41 &  80 &  64 &  16 \\
LIPPERT2 & 41 &  80 &  64 &  16 \\
\textbf{LUKVLI1} & 30 &  28 &  28 &   0 \\
LUKVLI10 & 30 &  28 &  14 &   0 \\
LUKVLI11 & 30 &  18 &   3 &   0 \\
LUKVLI12 & 30 &  21 &   6 &   0 \\
LUKVLI13 & 30 &  18 &   3 &   0 \\
\textbf{LUKVLI14} & 30 &  18 &  18 &   0 \\
LUKVLI15 & 30 &  21 &   7 &   0 \\
LUKVLI16 & 30 &  21 &  13 &   0 \\
\textbf{LUKVLI17} & 30 &  21 &  21 &   0 \\
\textbf{LUKVLI18} & 30 &  21 &  21 &   0 \\
LUKVLI2 & 30 &  14 &   7 &   0 \\
\textbf{LUKVLI3} & 30 &   2 &   2 &   0 \\
LUKVLI4 & 30 &  14 &   4 &   0 \\
\textbf{LUKVLI6} & 31 &  15 &  15 &   0 \\
LUKVLI8 & 30 &  28 &  14 &   0 \\
\textbf{LUKVLI9} & 30 &   6 &   6 &   0 \\
MANNE & 29 &  20 &  10 &   0 \\
\textbf{MOSARQP1} & 36 &  10 &  10 &   0 \\
\textbf{MOSARQP2} & 36 &  10 &  10 &   0 \\
\textbf{NGONE} & 29 & 134 & 106 &   0 \\
\textbf{NUFFIELD} & 38 & 138 &  28 &   0 \\
OPTMASS & 36 &  30 &   6 &  24 \\
POLYGON & 28 & 119 &  94 &   0 \\
POWELL20 & 30 &  30 &  15 &   0 \\
READING4 & 30 &  60 &  30 &   0 \\
SINROSNB & 30 &  58 &  29 &   0 \\
\textbf{SVANBERG} & 30 &  30 &  30 &   0 \\
VANDERM1 & 30 &  59 &  29 &  30 \\
VANDERM2 & 30 &  59 &  29 &  30 \\
VANDERM3 & 30 &  59 &  29 &  30 \\
VANDERM4 & 30 &  59 &  29 &  30 \\
\textbf{YAO} & 30 &  30 &   1 &   0 \\
ZIGZAG & 28 &  30 &   5 &  20 \\\hline
\end{tabular}
\end{center}
\end{minipage}
\end{center}
\caption{Test set selected from the CUTEst collection. Parameters $n$, $m$, $\bar{m}$, and $p$ denote, respectively, the number of variables, of inequality constraints, of inequality constraints treated by the logarithmic barrier, and of equality constraints for the given problem.}
\label{test_set}
\end{table}}

\subsection{Performance and data profiles}

Numerical experiments will be analyzed using performance~\cite{Dolan_More_2002} and data~\cite{More_Wild_2009} profiles. To provide a brief overview of these tools, consider a set of solvers $\mathcal{S}$ and a set of problems $\mathcal{P}$. Let $t_{p,s}$ represent the number of function evaluations required by solver $s$ to satisfy the convergence test adopted for problem $p$.

For an accuracy $\tau=10^{-k}$, where $k\in\{1,3,5\}$, the following convergence test was adopted:
\begin{equation}
\label{eq:convergence_test}
f_M - f(\vecx) \geq (1-\tau)(f_M-f_L),
\end{equation}
where $f_M$ represents the objective function value of the worst feasible point determined by all solvers for problem $p$, and $f_L$ is the best objective function value obtained by all solvers, corresponding to a feasible point of problem $p$.

The convergence test given by \eqref{eq:convergence_test} requires a significant reduction in the objective function value by comparison with the worst feasible point $f_M$. For profiles computation, the objective function value is set to infinity at points that violate the feasibility condition, defined by $c(\vecx)>\epsilon_{\tt viol}$, {where $\epsilon_{\tt viol}$ takes the values $\{10^{-4}, 10^{-6}, 10^{-8}\}$ in the experiments below}, and
$$c(\vecx) = {\sum_{l=1}^m \max\{g_l(\vecx),0\}} + \sum_{j=1}^p |h_j(\vecx)|.$$

The performance of solver $s\in\mathcal{S}$ is measured by the fraction of problems in which the performance ratio is at most $\alpha$, given by:
\begin{equation*}
\rho_s(\alpha) =\dfrac{1}{|P|}\left|\left\{p \in P\mid
\dfrac{t_{p, s}}{\min \left\{t_{p, s^{\prime}}: s^{\prime} \in
S\right\}} \leq \alpha\right\}\right|.
\end{equation*}
A performance profile provides an overview of how well a solver performs across a set of optimization problems. Particularly relevant is the value $\rho_s(1)$, that reflects the efficiency of the solver, i.e., the percentage of problems for which the algorithm performs the best. Robustness, as the percentage of problems that the algorithm is able to solve, can be perceived for high values of $\alpha$.

Data profiles focus on the behavior of the algorithm during the optimization process. A data profile measures the percentage of problems that can be solved (given the accuracy $\tau$) with $\kappa$ estimates of simplex gradients and it is defined by:
\begin{equation*}
d_s(\kappa) =\dfrac{1}{|P|}\left|\left\{p \in P\mid t_{p, s} \leq \kappa\left(n_p+1\right)\right\}\right|,
\end{equation*}
where $n_p$ represents the dimension of problem $p$.

\subsection{Results and analysis}
\label{subsec:cutest}
This subsection aims to demonstrate the good numerical performance
of the LOG-DS algorithm, when compared against SID-PSM and
other state-of-the-art solvers.

\subsubsection{Strategies for addressing linear constraints}\label{sec:linear_comparison}
In this subsection, our focus lies on evaluating the performance
of LOG-DS using two distinct approaches for managing linear
inequality constraints, other than bounds. The first treats each linear inequality constraint (other
than bounds) as a general inequality, i.e. addressing it
within the proposed penalty-barrier approach. The second approach
consists on the use of an extreme barrier method, adjusting the
directions{,} so that Assumption~\ref{condition:geometry_ass}
is satisfied for the linear constraints (including
bounds).

It is important to understand the relevance of comparing
the two strategies. The logarithmic barrier approach is well known
for handling linear constraints very efficiently, especially in
the presence of a large number thereof. Moreover, conforming the
directions to the nearby linear constraints might affect the
geometry of the generated points, impacting the quality of the
surrogate models built to improve the performance of the
algorithm. Using the penalty approach allows us to keep using the
coordinate directions, which are known to have good geometry for
building linear models. Any other orthonormal basis could be used,
though the coordinate directions additionally conform to bound
constraints on the variables, so that we can treat them separately
from the penalty approach.

The {works} by Lucidi et al.~\cite{Lucidi_Sciandrone_Tseng_2002} and
Lewis and Torczon~\cite{Lewis_Torczon_2000} propose methods for
computing directions conforming to linear inequality constraints
but do not consider degeneracy. Abramson et
al.~\cite{Abramson_Brezhneva_DennisJr_Pingel_2008} provide a
detailed algorithm for generating an adequate set of directions, regardless of whether the constraints are
degenerate or not.

In order to use Definition \ref{eps_active_set}, of $\varepsilon$-active constraints, a preliminary scaling of the constraints is assumed. For this purpose, the $i$-th constraint is multiplied by $\|\bvec{a}_i\|^{-1}$, since the vectors $\bvec{a}_i$, $i\in \I_{\X}(\vecx_k)$, are not null (note that $\|\bvec{a}_i\|\neq 0$ for all $i\in \I_{\X}(\vecx_k)$). Therefore, we consider

\begin{equation}
\overline{\bvec{a}}_i=\dfrac{\bvec{a}_i}{\|\bvec{a}_i\|},\quad
\overline{b}_i=\dfrac{b_i}{\|\bvec{a}_i\|},\ i\in
\I_{\X}(\vecx_k). \label{scaling}
\end{equation}
The $\varepsilon$-active index set is computed using the matrix
$\overline{\mA}$, a scaled version of the matrix $\mA$, and the
vector $\bvec{\overline{b}}$. Consequently,
$\|\overline{\bvec{a}}_i\|=1$ for all $i\in \I_\X(\vecx_k)$
and $\X=\{\vecx\in\mathbb{R}^n\mid \mA\vecx\leq
\bvec{b}\}=\{\vecx\in\mathbb{R}^n\mid \overline{\mA}\vecx\leq
\overline{\bvec{b}}\}$.

To compute the set of directions $\Dd_k$ that conforms to the geometry of the nearby constraints, the algorithm proposed in \cite[Alg. 4.4.2]{Abramson_Brezhneva_DennisJr_Pingel_2008} is used. The latter is divided into two parts: the first constructs the index set corresponding to the $\varepsilon$-active non-redundant constraints, and the second computes the set of directions $\Dd_k$, which include the generators of the cone $\T_\X(\vecx_k,\varepsilon)$. For identifying a constraint as being $\varepsilon$-active, $\varepsilon=10^{-5}$ was considered.  When linear constraints (other than bounds) are included in the merit function,  $\Dd_k = [\bvec{1}\; -\bvec{1}\;\quad \mathbb{I}_n\; -\mathbb{I}_n]$ was used as the set of poll directions for every $k$, where $\mathbb{I}_n$ is the set of columns of the identity matrix $\bvec{I}_n$ and $\bvec{1}$ stands for a vector of $n$ ones. This corresponds to the default set of poll directions in SID-PSM.

Figure~\ref{GSS_ABDP_vs_cord_all_problems_2k} depicts the
performance of LOG-DS using the two strategies described before,
considering a maximum number of $2000$ function evaluations, a
minimum step-size tolerance equal to $10^{-8}$, and a maximum feasibility violation $\epsilon_{\tt viol} = 10^{-4}$.

\begin{figure}[ht!]
\centering
\includegraphics[scale=0.4]{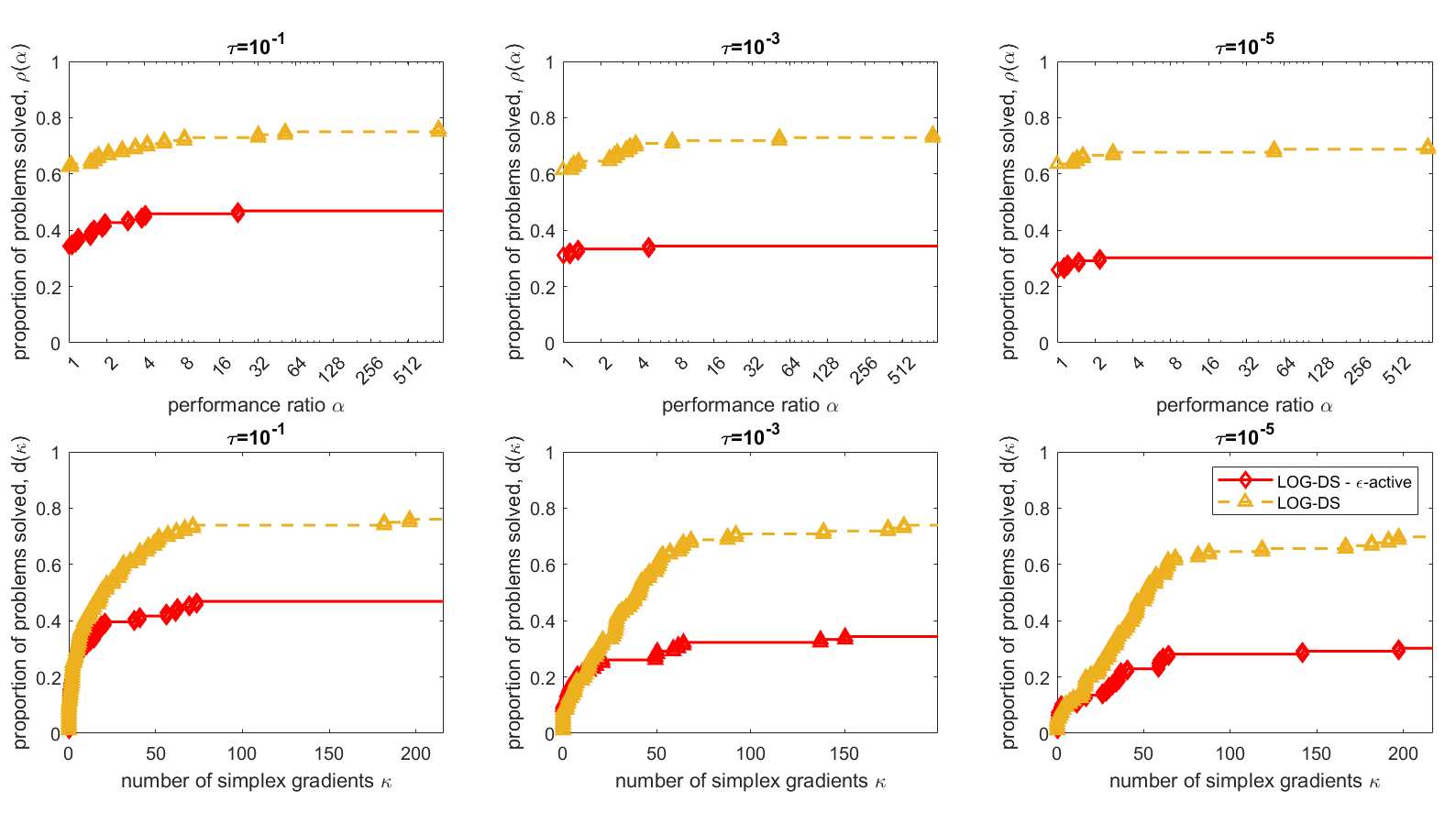}
\caption{Performance (on top) and data (on bottom) profiles comparing LOG-DS using two different approaches to address linear inequality constraints.} \label{GSS_ABDP_vs_cord_all_problems_2k}
\end{figure}


As it can be seen, the performance of LOG-DS algorithm when addressing the linear inequality constraints within the penalty approach outperforms the competing strategy, addressing the linear constraints directly. 
Therefore, in the rest of the work, the experiments will be carried out using the winning strategy. Note that the significant difference in the performance might be due to the specific choice of the tested problems and/or the specific strategy used to conform the directions to the linear constraints, rather than to a flaw in the approach by itself.

\subsubsection{Comparison with the original SID-PSM algorithm}
In the present work, an alternative strategy is proposed to address constraints within the SID-PSM algorithm. Thus, the discussion begins by illustrating that the use of a mixed penalty-logarithmic barrier is competitive against the extreme barrier approach. The latter is typically used for problems without equality constraints and for which a feasible point is given as initialization. So, a subset of problems satisfying these conditions was selected, consisting of a total of $28$ instances, highlighted in Table~\ref{test_set}.

Figure~\ref{log_vs_eb_2k} presents the comparison between LOG-DS, which exploits the mixed penalty-logarithmic barrier, and the original SID-PSM, which employs an extreme barrier approach. The default values of SID-PSM were considered for both algorithms, allowing a maximum number of $2000$ function evaluations, a
minimum step-size equal to $10^{-8}$, and a maximum feasibility violation $\epsilon_{\tt viol} = 10^{-4}$.
\begin{figure}[ht!]
\centering
\hspace*{-2cm}\includegraphics[scale=0.4]{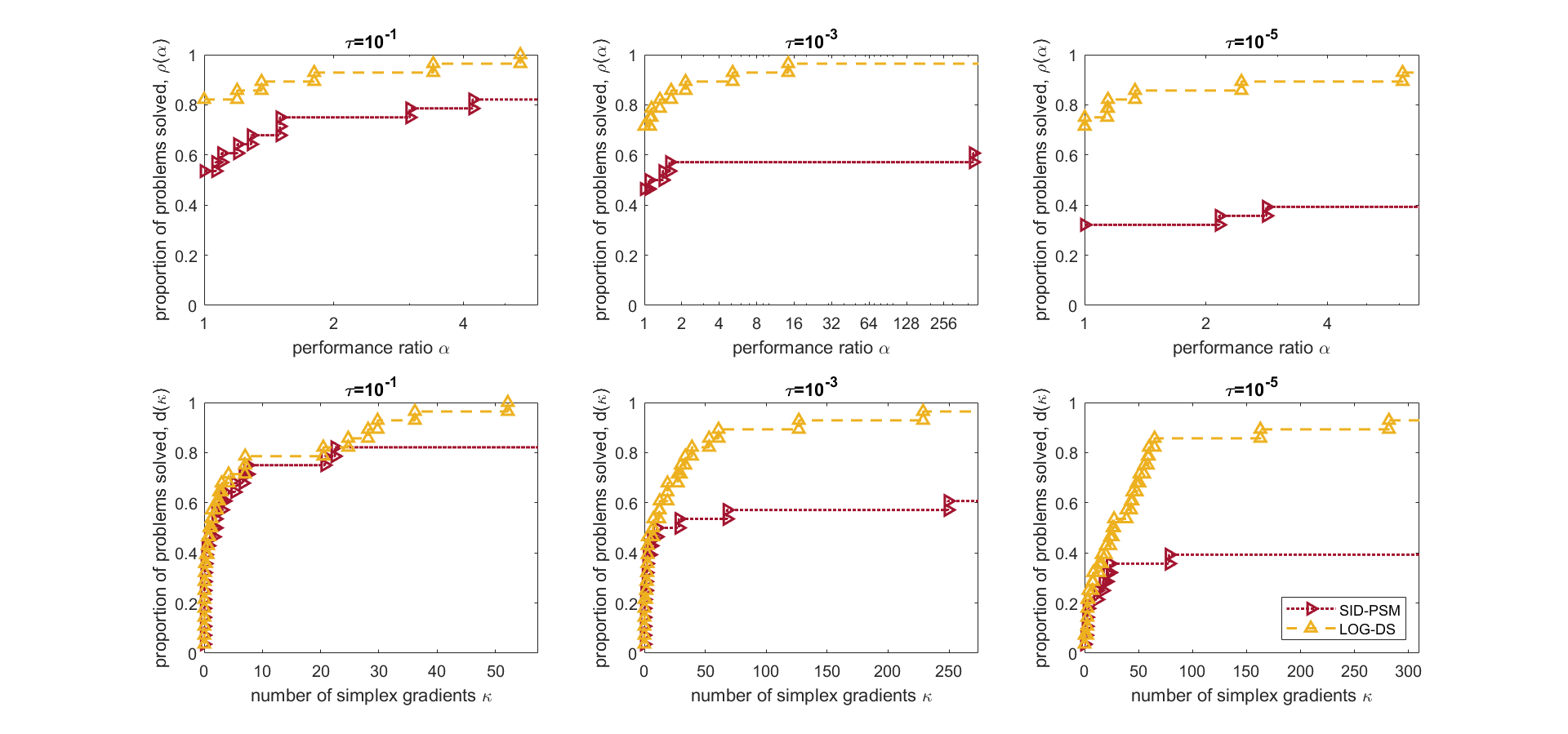}
\caption{Performance (on top) and data (on bottom) profiles comparing LOG-DS and SID-PSM.} \label{log_vs_eb_2k}
\end{figure}

As Figure~\ref{log_vs_eb_2k} shows, LOG-DS presents a better performance than SID-PSM, especially when a higher precision is considered. Furthermore, the possibility of initializing LOG-DS with infeasible points{, regarding some of the constraints,} allows to handle a wider class of practical problems.

\subsubsection{Comparison with state-of-the-art solvers}
This subsection focuses on comparing LOG-DS against state-of-the-art derivative-free optimization solvers that are able to address general nonlinear constraints. Comparisons were made with MADS~\cite{Audet_DennisJr_2006}, implemented in the well-known NOMAD package~\cite{Audet_LeDigabel_Montplaisir_Tribes_2022}(version 4), which can be freely obtained at~\url{https://www.gerad.ca/en/software/nomad}. Additionally, the 
X-LOG-DFL algorithm~\cite{Brilli_Liuzzi_Lucidi_2021}, available through the DFL library as the LOGDFL package at~\url{https://github.com/DerivativeFreeLibrary/LOGDFL}, was also tested. Comparison with LOG-DFL~\cite{Brilli_Liuzzi_Lucidi_2021} is particularly relevant since it uses the same merit function of LOG-DS. Default settings were considered for all solvers, with a maximum budget of $2000$ function evaluations and a maximum feasibility violation $\epsilon_{viol}=10^{-4}$. 

The first numerical profiles are reported in Figure~\ref{all_and_feas_2k}, where it can be observed that LOG-DS presents the best performance, for any of the three precision levels considered, both in terms of efficiency and robustness, across the different computational budgets. 


\begin{figure}[ht!]
\centering 
{\label{all_problems_2k}\includegraphics[scale=0.4]{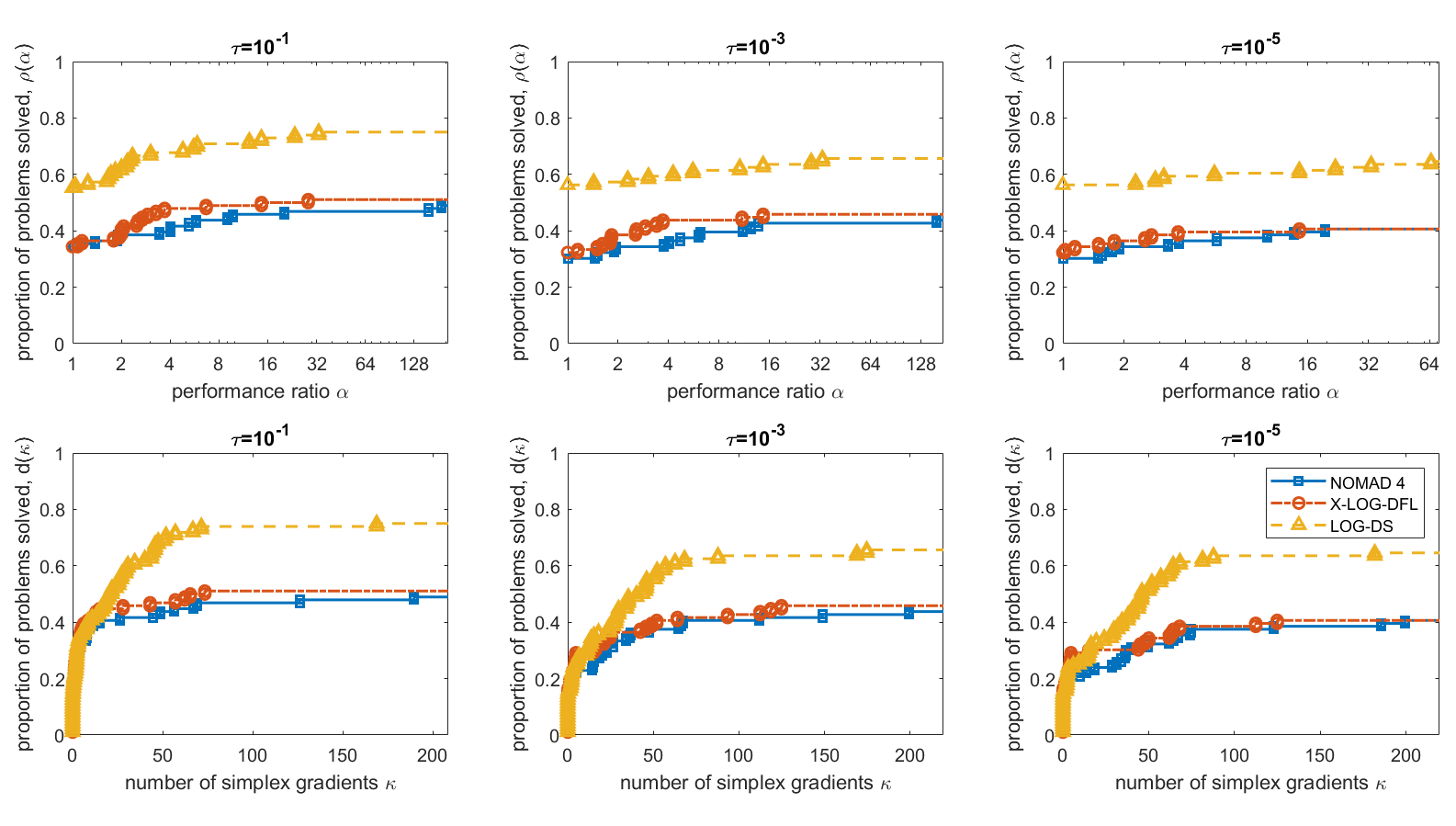}}
\caption{Performance (on top) and data (on bottom) comparing LOG-DS, NOMAD, and X-LOG-DFL on the complete problem collection.} \label{all_and_feas_2k}
\end{figure}

Theoretical guarantees for problems with equality constraints have not yet been established for NOMAD. To {keep} fairness in the experiments, Figure~\ref{noneq_2k} reports the results on the subset of problems with only inequality constraints ($61$ out of $96$ problems). Once again, LOG-DS appears to be the solver with the best performance.

\begin{figure}[ht!]
\centering
\includegraphics[scale=0.4]{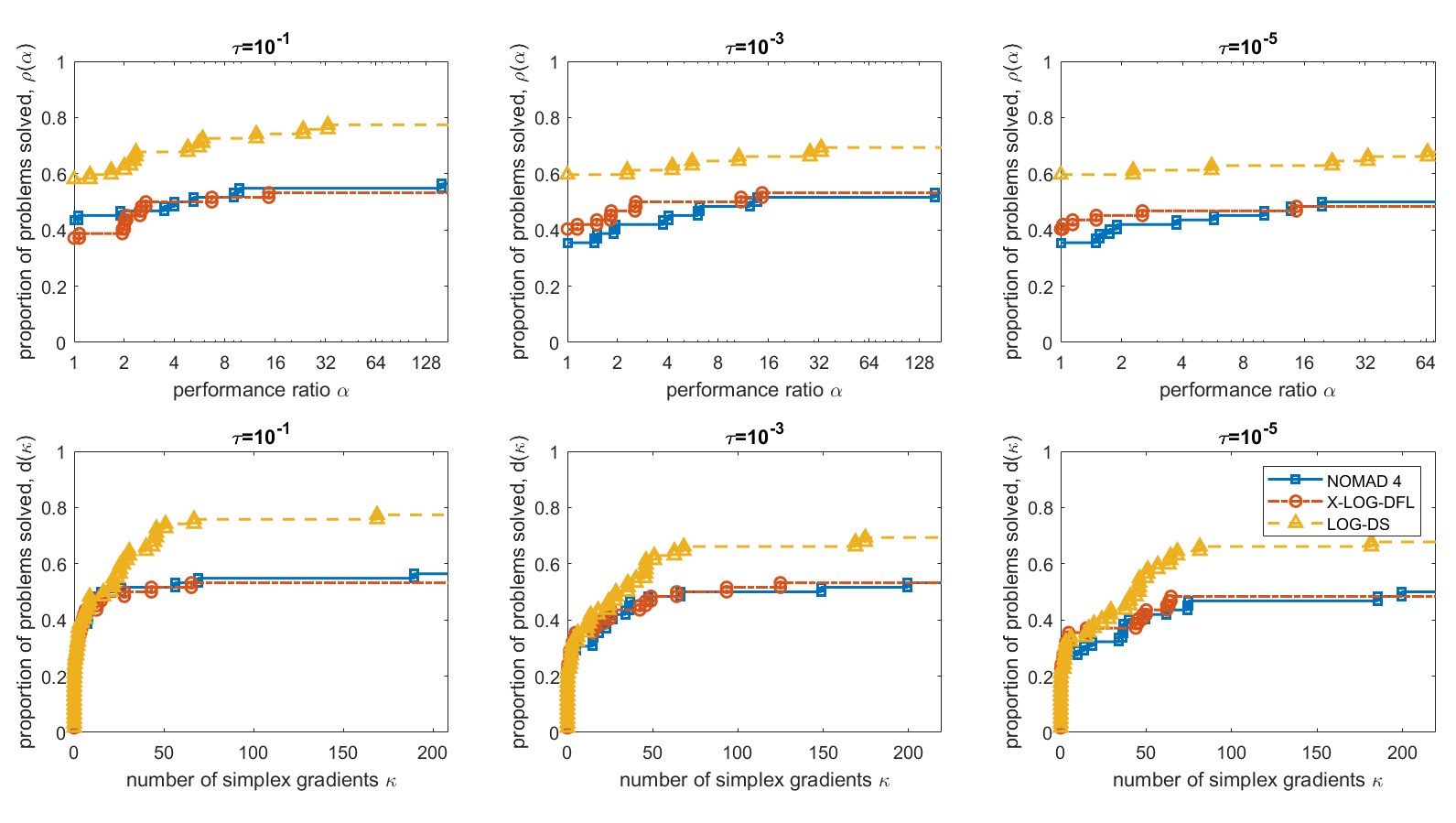}
\caption{Performance (on top) and data (on bottom) profiles comparing LOG-DS, NOMAD, and LOG-DFL on the subset of problems with only inequality constraints.}
\label{noneq_2k}
\end{figure}

Finally, {for the sake of transparency}, Figures~\ref{viol_6} and~\ref{viol_8} report the comparison between LOG-DS and NOMAD, considering lower levels of relaxation for the violation of the nonlinear constraints, namely $\epsilon_{\tt viol} = 10^{-6}$ and $\epsilon_{\tt viol} = 10^{-8}$. {The results are consistent with those presented for $\epsilon_{\tt viol}=10^{-4}$.}




\begin{figure}[h!]
\centering
\hspace*{-0.5cm}\subfigure[$\epsilon_{\tt viol}=10^{-6}$ \label{viol_6}]{
\includegraphics[width=0.5\textwidth]{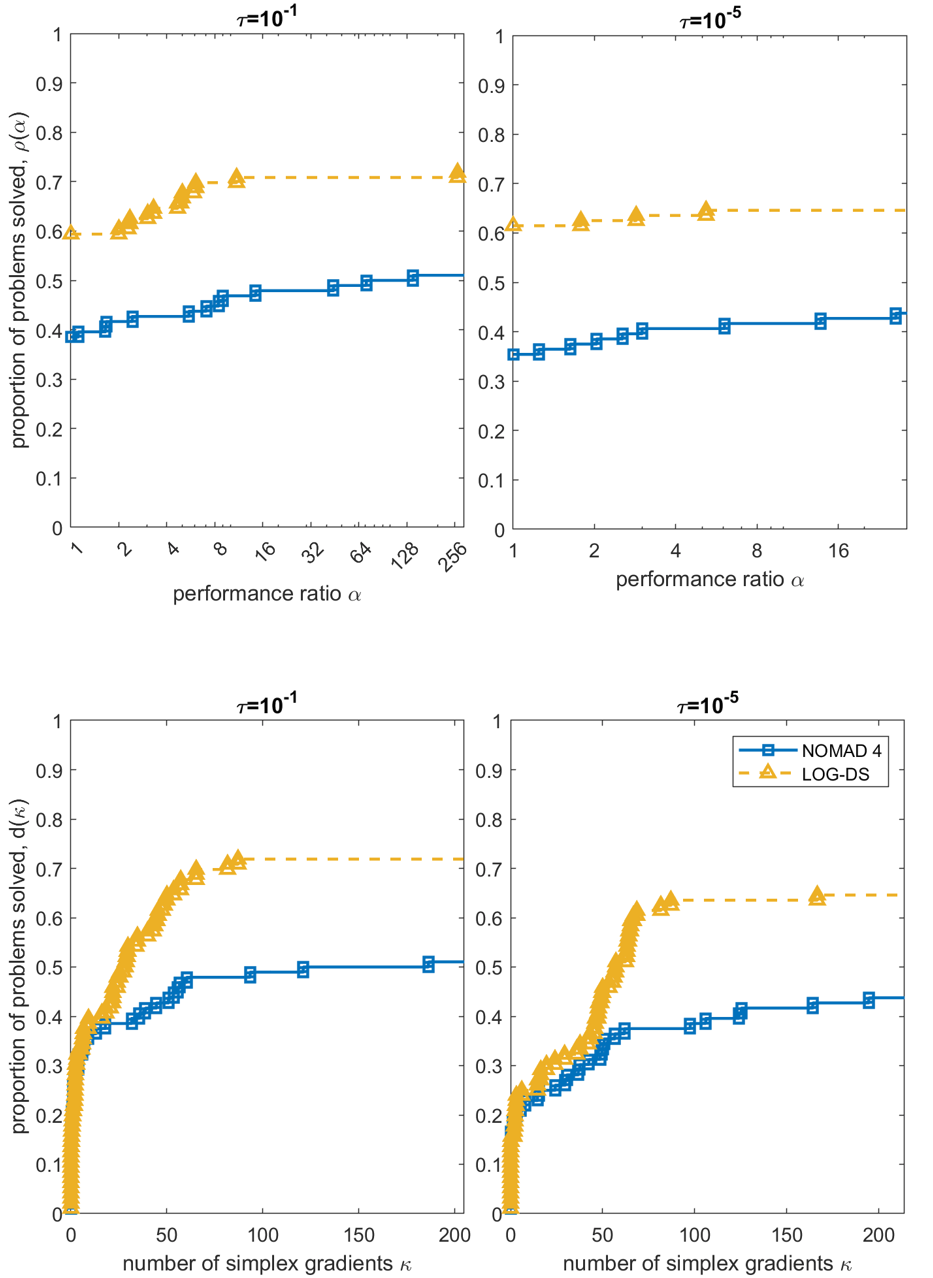}
}\hfill
\hspace*{-0.5cm}\subfigure[$\epsilon_{\tt viol}=10^{-8}$ \label{viol_8}]{
\includegraphics[width=0.5\textwidth]{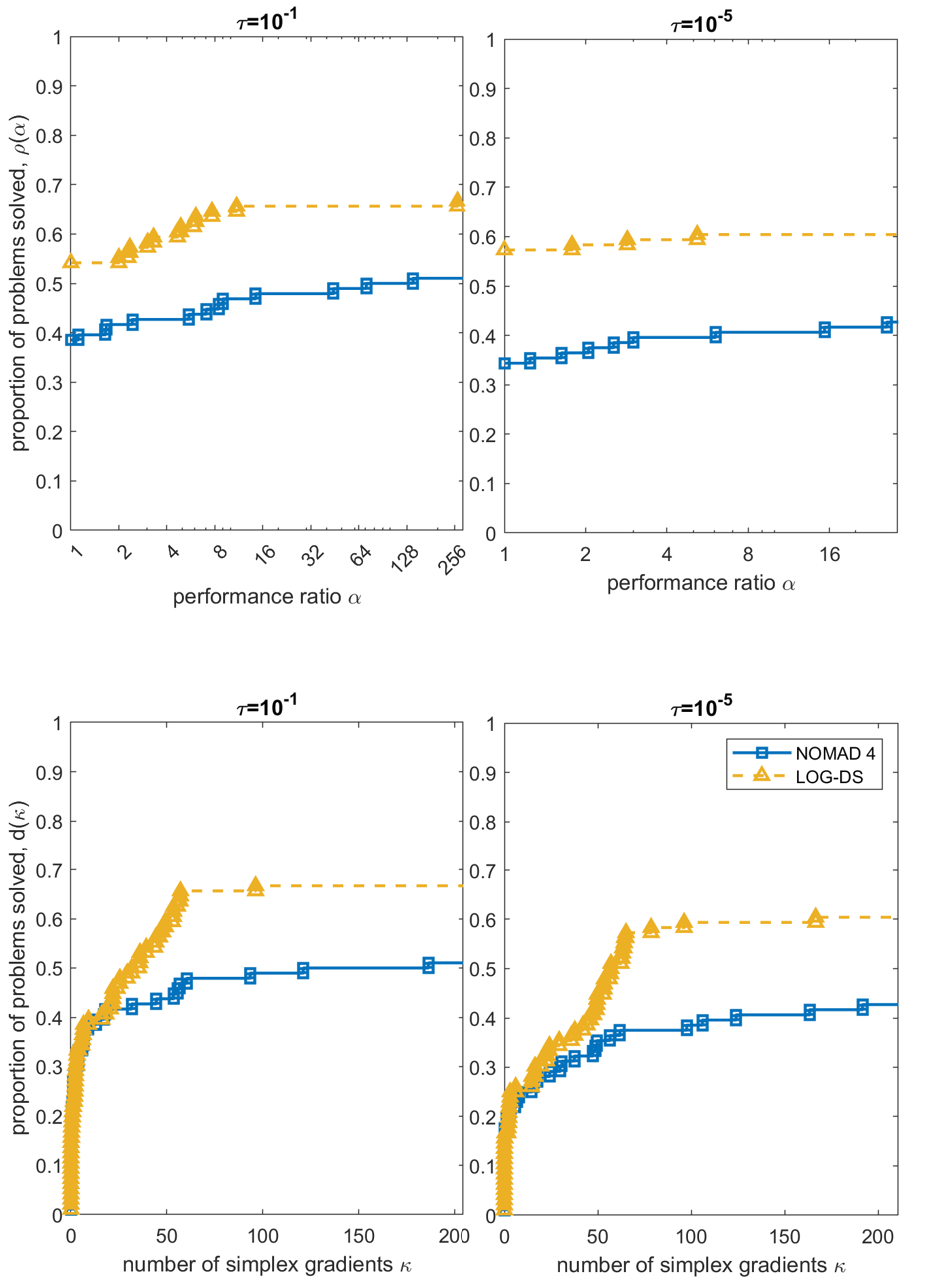}}
\caption{{Performance (on top) and data (on bottom) profiles comparing LOG-DS and NOMAD on the complete problem collection, for different levels of violation of the nonlinear constraints.}}
\label{fig:E2}
\end{figure}

In summary, considering the outcomes of the different numerical experiments, the performance of LOG-DS on the CUTEst collection of problems appears to be the most efficient and robust across different scenarios, demonstrating the numerical applicability of the algorithm.

\subsection{An engineering application}
The numerical performance of LOG-DS is now evaluated in a real-world engineering benchmark, focused on the design and operation of a concentrated solar power (CSP) plant, as detailed in~\cite{solar_paper}.  


The SOLAR benchmark comprises realistic black-box optimization problems, simulating both design and off-design scenarios for a molten-salt CSP plant with thermal energy storage. It highlights the main challenges of black-box optimization, including complex and diverse problem structures, while remaining reproducible and openly accessible.

The tests are conducted in the {\sf solar 6} instance, which consists of a problem with five variables and six nonlinear inequality constraints. In the official {\sf solar} repository\footnote{\url{https://github.com/bbopt/solar/blob/master/tests/6_MINCOST_TSTS/x0\_LH30.txt}}, thirty different starting points are provided, therefore the experiments are performed on the {corresponding} thirty different instances of the problem. 

LOG-DS and NOMAD are compared under {a function evaluation budget of $200n$}, with no relaxation allowed on the violation of the constraints, i.e. $\epsilon_{\tt viol}=0$. Results are presented in two subsections, each corresponding to a different version of LOG-DS. The first version uses the same algorithmic settings as in {the} previous experiments on academic test problems. The second considers a different approach for generating the poll directions, motivated by the characteristics of the {\sf solar} problem. Specifically, the functions in {\sf solar 6} are {nondifferentiable} and exhibit nonsmoothness or discontinuities. As a result, relying solely on coordinate directions and vectors of ones or minus ones may limit the achievable solution precision, and even surrogate models may fail to provide further improvement. {To tackle this issue, a new strategy to build the set of poll directions is introduced in Subsection~\ref{sec532}.}

\subsubsection{LOG-DS with default poll directions}
{Figure~\ref{SOLAR6_coord_1000fun}  illustrates} the results obtained with LOG-DS and NOMAD.
The profiles show that LOG-DS with the default set of poll directions, considered in the previous tests, achieves lower final accuracy on the SOLAR instances and terminates earlier than NOMAD. A plausible explanation is that some of the functions are nonsmooth, and possibly discontinuous, near the optimum. The models cannot produce consistent descent directions or meaningful improvements, so the stopping criteria of LOG-DS are triggered prematurely. By contrast, NOMAD’s direct search mechanism is less dependent on smooth local models and can continue to probe the {neighborhood}, which explains its superior final precision on this problem. This suggests that for the SOLAR benchmark one should avoid relying solely on coordinate or model-based directions and instead use polling strategies based on randomized or enlarged direction sets that are more robust to nonsmoothness.


\begin{figure}[h!]
\centering
\includegraphics[scale=0.4]{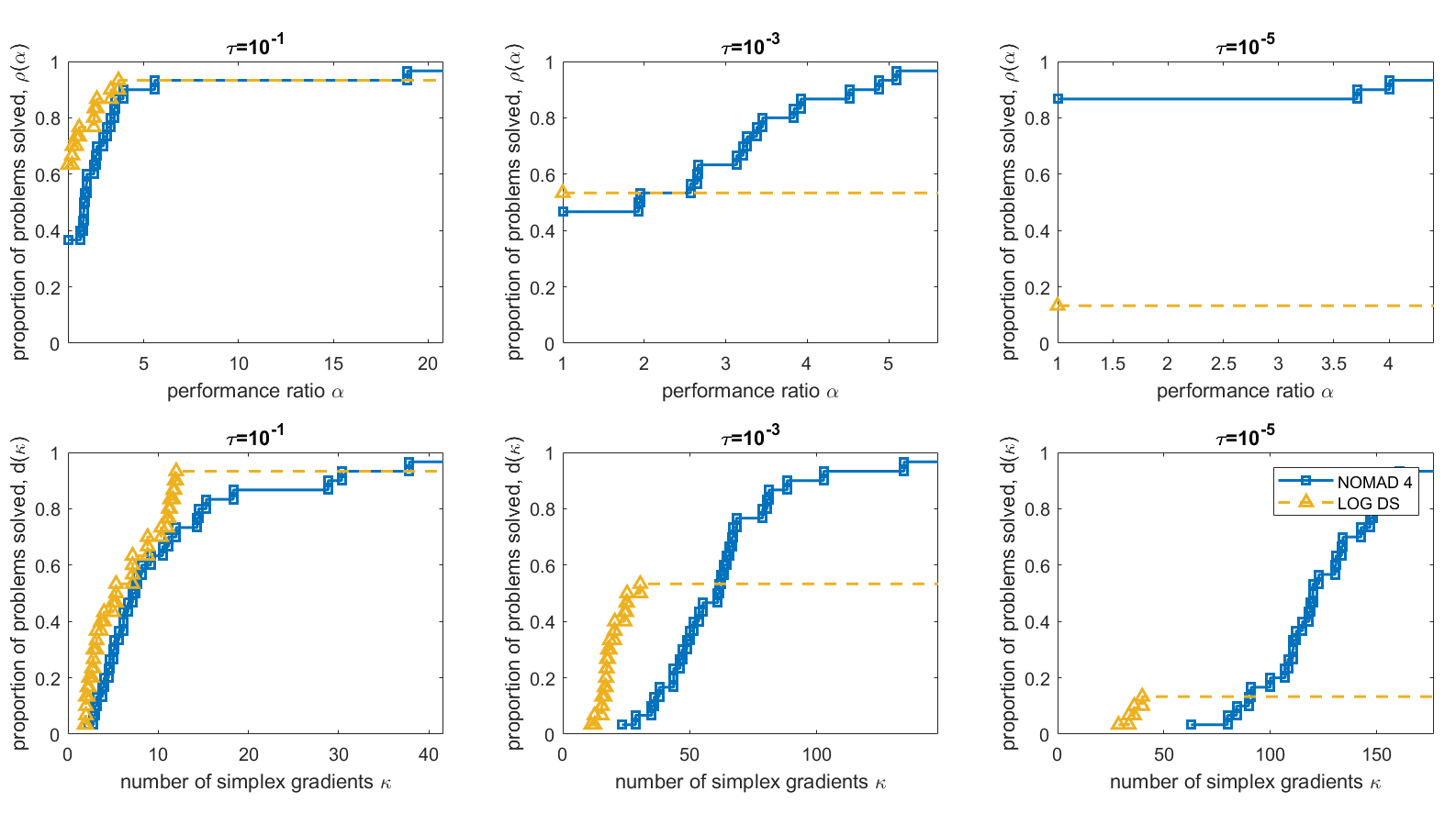}
\caption{Performance (on top) and data (on bottom) profiles comparing LOG-DS with the default set of poll directions and NOMAD on the instance {\sf solar 6}, allowing $1000$ function evaluations.}
\label{SOLAR6_coord_1000fun}
\end{figure}

\subsubsection{LOG-DS with quasi-dense structured directions}\label{sec532}
Motivated by the previous subsection, a different way of generating the poll directions is proposed. At each iteration $k$, the algorithm generates a random vector {$\bvec{v}_k\in\R^n$}, whose entries are independent and identically distributed. {Every entry is either 0 with probability $p = 0.5$, or follows a standard normal distribution with probability $1 - p$.} This construction yields a random vector with expected sparsity {$p$}, i.e., on average, half of the entries are exactly zero, while the nonzero entries are standard Gaussian random variables. The Householder transformation {$\bvec{H}_k=\bvec{I}_n-2\bvec{v}_k\bvec{v}_k^\top$} is then applied, for orthogonality. {A similar use of the Householder transformation for orthogonality has been proposed in OrthoMADS~\cite{Abramson_Audet_Dennis_LeDigabel_2009}}. The columns of $\bvec{H}_k$ form a basis for $\R^n$, and, crucially, the coordinate directions corresponding to the zero entries of {$\bvec{v}_k$} are preserved among the columns of $\bvec{H}_k$. This ensures that the generated set of directions includes both random directions and coordinate-aligned directions associated with the sparsity pattern of {$\bvec{v}_k$}.

To form a positive spanning set, $\Dd_k$ is set to include both the columns of 
$\bvec{H}_k$ and their negatives, and each direction is normalized. This approach provides a structured yet randomized set of directions that enhances the algorithm’s ability to explore the local neighborhood, particularly in the presence of nonsmoothness or discontinuities, while guaranteeing that coordinate directions are always represented when the corresponding entry of {$\bvec{v}_k$} is zero.

The results of the experiments are reported {in Figure~\ref{SOLAR_1000fun}}. It can be observed that {this new variant of } LOG-DS presents the best performance, for any of the three precision levels considered, both in terms of efficiency and robustness. {Let us highlight that LOG-DS reaches convergence on most of the instances using only half of the available budget.}


\begin{figure}[h!]
\centering
\includegraphics[scale=0.4]{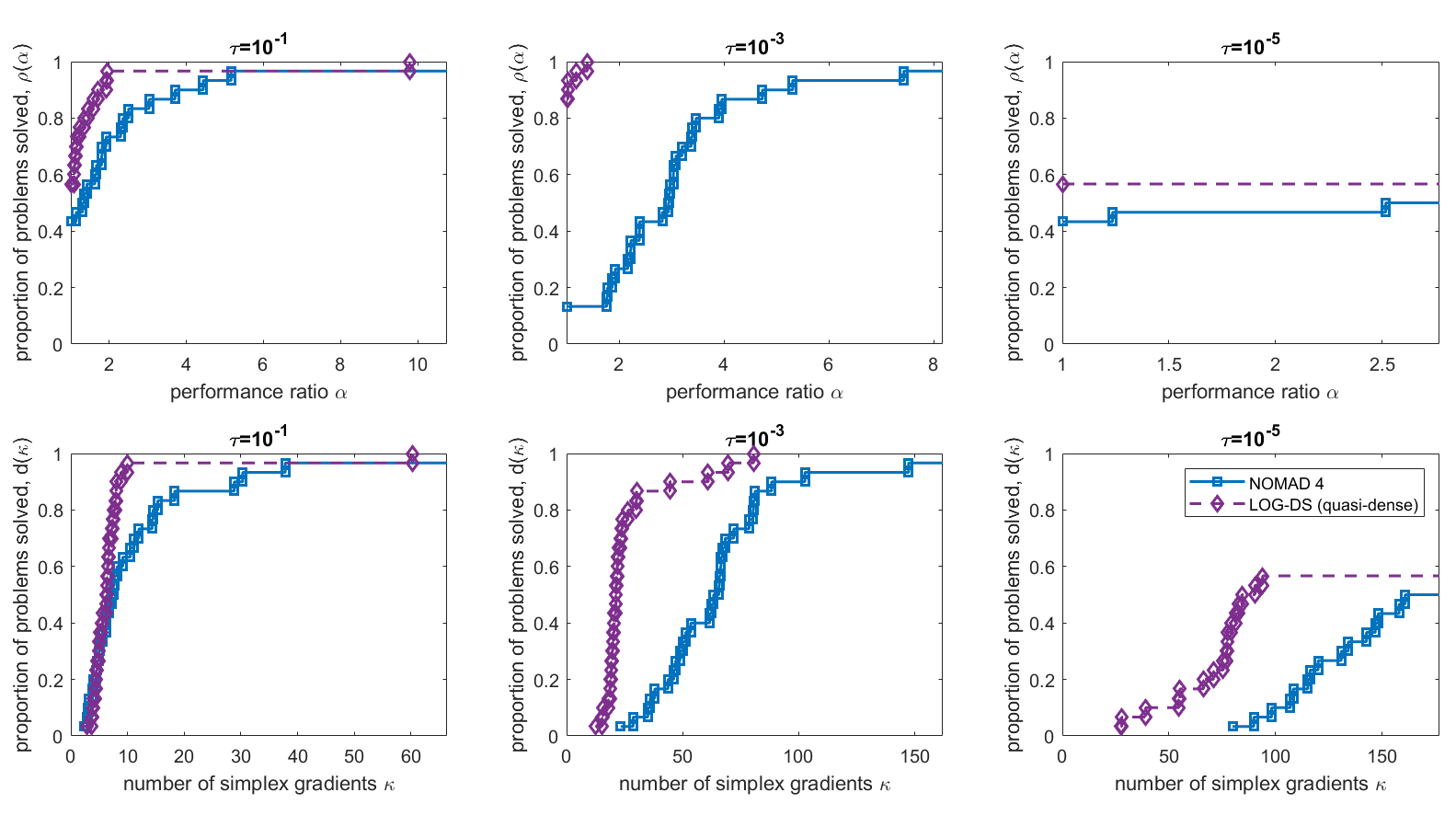}
\caption{Performance (on top) and data (on bottom) profiles comparing LOG-DS with quasi-dense directions and NOMAD on the instance {\sf solar 6}, allowing $1000$ function evaluations.}
\label{SOLAR_1000fun}
\end{figure}

\section{Conclusions}
\label{conclusions}
The primary objective of this work was to extend the approach introduced in~\cite{Brilli_Liuzzi_Lucidi_2021} to generalized pattern search, allowing to efficiently address nonlinear constraints.

{To this end, the present work adapted the SID-PSM algorithm, a generalized pattern search method in which polynomial models are used both at the search and the poll steps to improve the numerical performance. The resulting algorithm, LOG-DS, retains the basic algorithmic features of SID-PSM but uses a mixed penalty-logarithmic barrier merit function to address general nonlinear and linear constraints.}


Under standard assumptions, not requiring convexity of the functions defining the problem, convergence was established towards stationary points. Furthermore, an extensive numerical experimentation allowed to compare the performance of LOG-DS with several state-of-the-art solvers on a large set of test problems from the CUTEst collection and an engineering application. The numerical results indicate the robustness, efficiency, and overall effectiveness of the proposed algorithm.

\paragraph{Acknowledgments}
The authors would like to thank Christophe Tribes, from GERAD and Polytechnique Montréal, for helping with the numerical experiments related to NOMAD and two anonymous reviewers for the valuable comments and suggestions that led to a better version of the paper.

\section*{Declarations}
\paragraph{Funding}
The second and fourth authors were funded by national funds through FCT - Funda\c c\~ao para a Ci\^encia e a Tecnologia I.P., under the scope of the projects UID/00297/2025 and UID/PRR/00297/2025 (Center for Mathematics and Applications). The fourth author was also funded by national funds through FCT - Funda\c c\~ao para a Ci\^encia e a Tecnologia I.P., under the scope of the project UI/BD/151246/2021.

\paragraph{Competing interests}
The authors have no competing interests to declare that are relevant to the content of this article.

\paragraph{Data availability}
The datasets generated and/or analyzed during the current study are available from the corresponding author on reasonable request.

\paragraph{Code availability}
Code developed in this study can be obtained from the corresponding author on reasonable request.

\appendix

\section{Two technical results}
In this {appendix}, two technical results involving inequalities for real numbers and exponents are presented. The first lemma
provides a bound for the sum of powers of positive real numbers.
\begin{lemma}\cite[Equation (2), p. 37]{Rudin_1991}
\label{lemma2_appendix} Let $a\in\R_+$, $b\in\R_+$, and $p \in (0,1]$. It results $$(a+b)^p \leq a^p + b^p.$$
\end{lemma}

The next result establishes bounds concerning the difference of powers of absolute values and maximum functions, which are fundamental for proving the boundedness of the Lagrange multipliers.
\begin{lemma}
\label{lemma3_appendix} Let $a\in\R$, $b\in\R$ and $p \in (0,1]$. The following inequalities hold
\begin{equation}
\label{lemma3_appendix_ineq1}
||a|^p-|b|^p|\leq |a-b|^p,
\end{equation}
and
\begin{equation}
\label{lemma3_appendix_ineq2}
|\max\{a,0\}^p-\max\{b,0\}^p| \leq |a - b|^p.
\end{equation}
\end{lemma}
\begin{proof}
To prove \eqref{lemma3_appendix_ineq1} consider
\begin{equation}
\label{eq:lemmaA3_1}
|a|^p = |a-b+b|^p \leq (|a-b|+|b|)^p \leq |a-b|^p+|b|^p,
\end{equation}
where the last inequality follows from Lemma \ref{lemma2_appendix}. On the other hand,
\begin{equation}
\label{eq:lemmaA3_2}
|b|^p = |b-a+a|^p \leq (|b-a|+|a|)^p \leq |a-b|^p+|a|^p.
\end{equation}
Then, from inequalities \eqref{eq:lemmaA3_1} and \eqref{eq:lemmaA3_2}, inequality~\eqref{lemma3_appendix_ineq1} is derived.\\
To prove \eqref{lemma3_appendix_ineq2}, four different cases are analyzed:
\begin{itemize}
\item[i)] If $a\leq 0$ and $b\leq 0$, the result holds trivially.
\item[ii)] If $a\leq 0$ and $b>0$, then $|-b^p|=b^p\leq |b-a|^p=|a-b|^p$.
\item[iii)] If $a>0$ and $b\leq 0$, then ${|a^p|}=a^p\leq |a-b|^p$.
\item[iv)] If $a>0$ and $b>0$, using \eqref{lemma3_appendix_ineq1} we conclude that:
\begin{equation*}
|\max\{a,0\}^p-\max\{b,0\}^p| =|a^p-b^p|=||a|^p-|b|^p|\leq |a - b|^p.
\end{equation*}

\end{itemize}
\end{proof}

\section{Boundedness of sequences of Lagrange multipliers}
\label{appendix}
First, a result regarding a property of sequences of scalars is presented, which is used to prove the boundedness of the Lagrange multipliers.
\begin{lemma}
\label{lemma1_appendix} Let $\{a^i_k\}_{k\in\mathbb{N}}$,
$i=1,\ldots,\ell$, be sequences of scalars. 
Either all the sequences converge to zero, or there exists 
an index $s\in\{1,\ldots,\ell\}$ and an infinite subset $\K\subseteq
\mathbb{N}$ such that:
\begin{equation}
\Disp \lim_{k\to+\infty,\atop k\in \K}
\dfrac{a_k^i}{|a_k^{s}|}=z_i, \quad |z_i|<+\infty, \quad
i=1,\ldots,\ell.
\end{equation}


%
\end{lemma}

\begin{proof}
%
Let us assume that at least one sequence is not
convergent to zero. Then, an $\epsilon>0$ and an infinite index set $\hat{\K}$ exist such that for $k\in\hat{\K}$ and sufficiently large, $|a_k^{s_k}| = \max\{|a_k^i|\mid i=1,\ldots,\ell\}>\epsilon$, where $s_k\in\{1,\dots,\ell\}$. Since the number of sequences, $\ell$, is finite, it is possible to select a subsequence of $\hat\K$, indexed in $\bar{\K}$, such that  
$|a_k^{s_k}|=|a_k^{s}|>0$ for all $k\in\bar{\K}$. Thus, for $i=1,\ldots,\ell$, $\left\{\frac{a_k^i}{|a_k^{s}|}\right\}_{k\in\bar{\K}}$ is bounded between $-1$ and $1$. The result follows, since every bounded sequence has convergent subsequences.
\end{proof}

The next result establishes the boundedness of the subsequence of Lagrange multipliers.
\begin{theorem}\label{theo:bounded_multipliers}
In the conditions of Theorem~\ref{main_theorem}, the subsequences $\set{\lambda_\ell(\vecx_k;\rho_k)}_{k\in\K_\rho^{\tt x}}$, $\ell=1,\ldots,m$ and $\set{\mu_j(\vecx_k;\rho_k)}_{k\in\K_\rho^{\tt x}}$, $j=1,\ldots,p$, defined as:
\begin{eqnarray*}
&& \lambda_\ell(\vecx_k;\rho_k) = \begin{cases} \dfrac{\rho_k}{-g_\ell(\vecx_k)}, & \text { if } \ell\in\Gint\\ \nu\left(\dfrac{\max\{g_\ell(\vecx_k),0\}}{\rho_k}\right)^{\nu-1}, & \text { if } \ell\in\Gext \end{cases}\\
&& \mu_j(\vecx_k;\rho_k) =
\nu\left(\dfrac{|h_j(\vecx_k)|}{\rho_k}\right)^{\nu-1}, \quad
j=1,\ldots,p,
\end{eqnarray*}
are bounded.
\end{theorem}
\begin{proof}
Recalling the expressions of $\lambda_\ell(\vecx;\rho)$, $\ell=1,\ldots,m$, and of $\mu_j(\vecx;\rho)$, $j=1,\ldots,p$, inequality~\eqref{eq:grad_Z} can be rewritten as
\begin{equation}
\label{eq_ap:grad_Z_1}
\begin{split}
&\Bigg(\nabla f(\vecy_k^i) + \sum_{\ell=1}^{m} \lambda_\ell(\vecy_k^i;\rho_k) \nabla g_\ell(\vecy_k^i) + \\
&\qquad \sum_{j=1}^{p} \mu_j(\vecy_k^i;\rho) \nabla h_j
(\vecy_k^i) \Bigg)^{\top}{\vecd_k^i} \geq
-\frac{\xi(\hat{\alpha}_k^i)}{\hat{\alpha}_k^i}, \quad \forall\
i\in \J_k\ \mbox{and}\ k\in\K_\rho^{\tt x},
\end{split}
\end{equation}
where $\vecy_k^i=\vecx_k+t_k^i{\hat{\alpha}_k^i} \vecd_k^i$,
with $t_k^i\in (0,1)$ and $\hat{\alpha}_k^i\leq \alpha_k$.

Let us start by establishing that
\begin{equation}
\label{eq_ap:lim_lambda_log} \Disp \lim_{k\to+\infty\atop k\in
\K_\rho^{\tt x}} |\lambda_\ell(\vecx_k;\rho_k) - \lambda_\ell(\vecy_k^i;\rho_k) | = 0, \quad
\ell\in\Gint, \quad \forall\ i\in \J_k
\end{equation}
In fact,
\begin{eqnarray}
\label{eq_ap:in1_lambda}
\left|\frac{\rho_k}{-g_\ell(\vecx_k)}-\frac{\rho_k}{-g_\ell(\vecy_k^i)}\right|
&=&
\rho_k\left|\frac{g_\ell(\vecx_k)-g_\ell(\vecy_k^i)}{(-g_\ell(\vecy_k^i))(-g_\ell(\vecx_k))}\right|\\
&=& \rho_k \frac{\left|\nabla g_\ell(\vecu_k^{\ell})^\top
\left(\vecx_k-\vecy_k^i\right)\right|}{|g_\ell(\vecy_k^i)|
|g_\ell(\vecx_k)|}\nonumber \\ &\leq& \rho_k
\frac{\left\|\nabla
g_\ell(\vecu_k^{\ell})\right\|\left\|\vecy_k^i-\vecx_k\right\|}{|g_\ell(\vecy_k^i)
| |g_\ell\left(\vecx_k\right)|},\nonumber
\end{eqnarray}
where $\vecu_k^{\ell}= \vecx_k+\tilde{t}_k^{\ell}(\vecy_k^i-\vecx_k)$ with $\tilde{t}_k^{\ell}\in (0,1)$.

Then, there is $c_1>0$ such that
\begin{eqnarray}
\label{eq_ap:in2_lambda}
\rho_k \left\|\nabla
g_\ell(\vecu_k^{\ell})\right\|
\frac{\left\|\vecy_k^i-\vecx_k\right\|}{|g_\ell(\vecy_k^i)| |g_\ell\left(\vecx_k\right)|} &\leq& \rho_k c_1 \dfrac{\|\vecy_k^i-\vecx_k
\|}{|g_\ell(\vecy_k^i)| |g_\ell(\vecx_k)|}\\  &=& \rho_k c_1
\dfrac{\|\vecx_k+t_k^i\hat{\alpha}_k^i \vecd_k^i-\vecx_k
\|}{|g_\ell(\vecy_k^i)||g_\ell(\vecx_k)|}\nonumber \\ &=&  \rho_k c_1
\dfrac{\|t_k^i\hat{\alpha}_k^i \vecd_k^i
\|}{|g_\ell(\vecy_k^i)| |g_\ell(\vecx_k)|}\nonumber \\ &=& \rho_k c_1
\dfrac{t_k^i \hat{\alpha}_k^i \|\vecd_k^i
\|}{|g_\ell(\vecy_k^i)| |g_\ell(\vecx_k)|}.\nonumber
\end{eqnarray}
Now, we will prove that there is another constant $c_2>0$ such
that
\begin{equation}
\label{eq_ap:const2} \dfrac{1}{|g_\ell(\vecy_k^i)|}\leq
c_2\dfrac{1}{|g_\ell(\vecx_k)|}.
\end{equation}
Suppose, in order to arrive to a contradiction, that $c_2$ does not exist. This would imply that there  exists $\K_\rho^{\tt x,g}\subseteq\K_\rho^{\tt x}\subseteq\K_\rho$ such that
\begin{equation}
\label{eq_ap:lim_g} \Disp \lim_{k\to+\infty \atop k\in \K_\rho^{\tt x,g}}
\dfrac{\frac{1}{|g_\ell(\vecy_k^i)|}}{\frac{1}{|g_\ell(\vecx_k)|}} =
\lim_{k\to+\infty \atop k\in\K_\rho^{\tt x,g}}
\dfrac{|g_\ell(\vecx_k)|}{|g_\ell(\vecy_k^i)|}=+\infty.
\end{equation}
Let us consider the case where
\begin{equation*}
\Disp \lim_{k\to+\infty \atop k\in \K_\rho^{\tt x,g}} |g_\ell(\vecx_k)|=0.
\end{equation*}
Since $g_\ell(\vecx_k)<0$ and $g_\ell(\vecy_k^i)<0$ for all $k\in \K_\rho^{\tt x,g}$, by~\eqref{eq_ap:lim_g} there exists $\overline{k}\in\mathbb{N}$ such that, for all $k\geq \overline{k}$, $k\in\K_\rho^{\tt x,g}$, we have $$-g_\ell(\vecx_k)>-g_\ell(\vecy_k^i)=-g_\ell(\vecx_k+t_k^i\hat{\alpha}_k^i \vecd_k^i).$$ Using the Lipschitz continuity of $g_\ell$, $\ell=1,\ldots,m$ and the fact that $\|\vecd_k^i\|=1$ for all $i\in \J_k$, it follows $$-g_\ell(\vecx_k+t_k^i\hat{\alpha}_k^i \vecd_k^i)\geq-g_\ell(\vecx_k)-L_{g_\ell}\|t_k^i\hat{\alpha}_k^i \vecd_k^i\|=-g_\ell(\vecx_k)-L_{g_\ell}t_k^i\hat{\alpha}_k^i.$$
The definition of $\K_\rho$ guarantees that $$\alpha_{k+1}\leq\min\set{\rho_k^{\beta},(\gmink)^2}, \qquad \alpha_{k+1}=\theta_{\alpha}\alpha_k,$$ so that
\begin{eqnarray}
\label{alfa_k_leq}
\alpha_k\leq \dfrac{\min\set{\rho_k^{\beta},(\gmink)^2}}{\theta_{\alpha}}
\end{eqnarray}
Hence, since $\hat{\alpha}_k^i\leq \alpha_k$, we have $$-g_\ell(\vecx_k)-L_{g_\ell} t_k^i \hat{\alpha}_k^i\geq -g_\ell(\vecx_k)-L_{g_\ell} t_k^i \frac{1}{\theta_{\alpha}}(g_\ell(\vecx_k))^2,\quad \forall k\geq\overline{k},\ k\in \K_\rho^{\tt x,g}$$ Thus,
\begin{eqnarray*}
\Disp \lim_{k\to+\infty \atop k\in \K_\rho^{\tt x,g}}
\dfrac{-g_\ell(\vecx_k)}{-g_\ell(\vecy_k^i)}&=&\lim_{k\to+\infty \atop k\in \K_\rho^{\tt x,g}}
\dfrac{-g_\ell(\vecx_k)}{-g_\ell(\vecx_k+t_k^i\hat{\alpha}_k^i \vecd_k^i)}\\&\leq&
\lim_{k\to+\infty \atop k\in \K_\rho^{\tt x,g}}
\dfrac{-g_\ell(\vecx_k)}{-g_\ell(\vecx_k)-L_{g_\ell} t_k^i \frac{1}{\theta_{\alpha}}(g_\ell(\vecx_k))^2}= 1,
\end{eqnarray*}
which leads to a contradiction, proving~\eqref{eq_ap:const2}.

Now, by considering the other case
\begin{equation*}
\Disp \lim_{k\to+\infty \atop k\in \K_\rho^{\tt x,g}}
|g_\ell(\vecx_k)|=c<+\infty,
\end{equation*}
we have
\begin{eqnarray*}
\Disp \lim_{k\to+\infty \atop k\in \K_\rho^{\tt x,g}}
\dfrac{-g_\ell(\vecx_k)}{-g_\ell(\vecy_k^i)}=\lim_{k\to+\infty \atop k\in \K_\rho^{\tt x,g}}
\dfrac{-g_\ell(\vecx_k)}{-g_\ell(\vecx_k+t_k^i\hat{\alpha}_k^i
\vecd_k^i)}<+\infty.
\end{eqnarray*}
Again, this leads to a contradiction, proving~\eqref{eq_ap:const2}.

Hence, the existence of the constant $c_2>0$, \eqref{eq_ap:const2}, and recalling that $\hat{\alpha}_k^i\leq\alpha_k$, allow us to write
\begin{eqnarray*}
\dfrac{\rho_k c_1 t_k^i
\hat{\alpha}_k^i}{|g_\ell(\vecy_k^i)||g_\ell(\vecx_k)|}&\leq&\dfrac{\rho_k
c_1 c_2 t_k^i \alpha_{k}}{|g_\ell(\vecx_k)|^2}.
\end{eqnarray*}
The instructions of Step 3 imply that $\vecx_{k+1} = \vecx_k$, so that $\gmink = \Disp\min_{\ell\in\Gint} \{|g_\ell(\vecx_{k+1})|\} = \Disp\min_{\ell\in\Gint} \{|g_\ell(\vecx_{k})|\}$. Recalling \eqref{alfa_k_leq}, it follows $$\dfrac{\rho_k \alpha_k}{(\gmink)^2}\leq \dfrac{\rho_k}{\theta_{\alpha}}.$$ Then, recalling Theorem~\ref{convergence_eps},~\eqref{eq_ap:lim_lambda_log} is proved.\medskip \\

\noindent Furthermore,
\begin{equation}
\label{eq_ap:lim_lambda_ext} \Disp \lim_{k\to+\infty\atop k\in \K_\rho^{\tt x}} |\lambda_\ell(\vecx_k;\rho_k) - \lambda_\ell(\vecy_k^i;\rho_k) | = 0, \quad \ell\in\Gext, \quad \forall\ i\in \J_k
\end{equation}
\noindent can be established. In fact, since  $\|\vecd_k^i\|=1$ for all $i\in \J_k$,
{\small
\begin{eqnarray*}
\hspace*{-1.cm}\left|\dfrac{\nu}{\rho_k^{\nu-1}} \left(\max\{g_\ell(\vecx_k),0\}\right)^{\nu-1}\right.&-&\left.\dfrac{\nu}{\rho_k^{\nu-1}} \left(\max\{g_\ell(\vecy_k^i),0\}\right)^{\nu-1}\right|=\\[7pt]
& = & \dfrac{\nu}{\rho_k^{\nu-1}} \left|\max \{g_\ell(\vecx_k), 0\}^{\nu-1}-\max\{g_\ell(\vecx_k)+\nabla g_\ell(\vecu_k^{\ell})^\top(\vecx_k-\vecy_k^i), 0\}^{\nu-1}\right|\\
(\text{Lemma \ref{lemma3_appendix} -- \eqref{lemma3_appendix_ineq2} })&\leq&\dfrac{\nu}{\rho_k^{\nu-1}} \left|g_\ell(\vecx_k) -g_\ell(\vecx_k)-\nabla g_\ell(\vecu_k^{\ell})^\top(\vecx_k-\vecy_k^i)\right|^{\nu-1} \\
& = &\dfrac{\nu}{\rho_k^{\nu-1}} |\nabla g_\ell(\vecu_k^{\ell})^\top (\vecx_k-\vecy_k^i)|^{\nu-1} \leq \dfrac{\nu}{\rho_k^{\nu-1}} \left\|\nabla g_\ell(\vecu_k^{\ell})\right\|^{\nu-1}\left\|\vecx_k-\vecy_k^i\right\|^{\nu-1} \\
& \leq &c_3\dfrac{\nu}{\rho_k^{\nu-1}} \left\|\left(\vecx_k-(\vecx_k+t_k^i\hat{\alpha}_k^i \vecd_k^i) \right)\right\|^{\nu-1}= c_3\dfrac{\nu}{\rho_k^{\nu-1}} (t_k^i\hat{\alpha}_k^i)^{\nu-1} \\
&\leq & c_3\dfrac{\nu}{\rho_k^{\nu-1}} (\hat{\alpha}_k^i)^{\nu-1}  \leq c_3\nu\left(\dfrac{\alpha_k}{\rho_k}\right)^{\nu-1} \leq  c_3\nu \left(\frac{\rho_k^{\beta-1}}{\theta_{\alpha}}\right)^{\nu-1} \\&=& c_3\nu \theta_{\alpha}^{1-\nu} \rho_k^{(\beta-1)(\nu-1)},
\end{eqnarray*}}
\noindent where $\vecu_k^{\ell}= \vecx_k+\tilde t_k^{\ell}(\vecy_k^i-\vecx_k)$, with $\tilde t_k^{\ell}\in (0,1)$, and $c_3>0$. Thus, using $\beta>1$, $\nu\in(1,2]$, and recalling Theorem \ref{convergence_eps},  \eqref{eq_ap:lim_lambda_ext} is proved. Therefore,
\begin{equation}
\label{eq_ap:lim_lambda} \Disp \lim_{k\to+\infty\atop k\in \K_\rho^{\tt x}} |\lambda_\ell(\vecx_k;\rho_k) - \lambda_\ell(\vecy_k^i;\rho_k) | = 0, \quad \ell=1,\ldots,m, \quad \forall\ i\in \J_k
\end{equation}
\medskip\\
\noindent Let us now prove that
\begin{equation}
\label{eq_ap:lim_mu}
\Disp \lim_{k\to+\infty\atop k\in \K_\rho^{\tt x}}
|\mu_j(\vecx_k;\rho_k) - \mu_j(\vecy_k^i;\rho_k) | = 0, \quad
j=1,\ldots,p, \quad \forall i\in \J_k
\end{equation}
In fact, recalling that $\nu\in(1,2]$ so that $\nu-1\in(0,1]$, we have
\begin{eqnarray}
\label{eq_ap:in1_mu}
\hspace*{-0.5cm}\left|\nu\left|\frac{h_j(\vecx_k)}{\rho_k}\right|^{\nu-1}-\nu\left|\frac{h_j(\vecy_{k}^i)}{\rho_k}\right|^{\nu-1}\right| &=& \frac{\nu}{\rho_k^{\nu-1}} \left|\left|h_j(\vecx_k)\right|^{\nu-1}-\left| h_j(\vecx_k)+\nabla h_j(\vecu_k^{j})^\top (\vecy_{k}^i-\vecx_k)\right|^{\nu-1}\right| \nonumber\\[7pt]
(\text{Lemma \ref{lemma3_appendix} -- \eqref{lemma3_appendix_ineq1}})
&\leq&\frac{\nu}{\rho_k^{\nu-1}} \left|h_j(\vecx_k)-h_j(\vecx_k)-\nabla h_j(\vecu_k^{j})^\top (\vecy_{k}^i-\vecx_k)\right|^{\nu-1} \nonumber\\[7pt]
&=&\frac{\nu}{\rho_k^{\nu-1}} \left|\nabla h_j(\vecu_k^j)^\top (\vecy_{k}^i-\vecx_k)
\right|^{\nu-1}\nonumber\\ &\leq& \frac{\nu}{\rho_k^{\nu-1}} \left\|\nabla h_j(\vecu_k^j)
\right\|^{\nu-1}\left\|(\vecy_{k}^i-\vecx_k)\right\|^{\nu-1},
\end{eqnarray}
where $\vecu_k^j=\vecx_k+\tilde t_k^j (\vecy_{k}^i-\vecx_k)$, with $\tilde t_k^j\in (0,1)$
Now, recalling that $h_j$, $j=1,\ldots,p$ are continuously
differentiable functions, $\vecy_{k}^i=\vecx_k+t_k^i \hat
\alpha_k^i \vecd_k^i$, with $t_k^i\in (0,1)$,  and $\|\vecd_k^i\|=1$,
from~\eqref{eq_ap:in1_mu} and the fact that $\hat\alpha_k^i\leq
\alpha_k$ we can write
\begin{eqnarray*}
\left|\nu\left|\frac{h_j(\vecx_k)}{\rho_k}\right|^{\nu-1}-
\nu\left|\frac{h_j (\vecy_{k}^i)}{\rho_k}\right|^{\nu-1}\right|
&\leq&
\frac{\nu}{\rho_k^{\nu-1}} c_4 \left(t_k^i \hat\alpha_k^i \right)^{\nu-1} \\ &\leq& c_4\frac{\nu}{\rho_k^{\nu-1}}\left(t_k^i \alpha_k \right)^{\nu-1}\leq c_4\nu\left(\frac{\alpha_k}{\rho_k}\right)^{\nu-1}\\
&\leq& c_4\nu\left(\frac{\rho_k^{\beta-1}}{\theta_{\alpha}}\right)^{\nu-1} = c_4\nu\theta_{\alpha}^{1-\nu} \rho_k^{(\beta-1)(\nu-1)}.
\end{eqnarray*}
Given that $\beta > 1$, $\nu\in (1,2]$, and recalling Theorem \ref{convergence_eps}, it {can be concluded} that \eqref{eq_ap:lim_mu} holds.\\

Let us now prove the boundedness of the sequences $\set{\lambda_\ell(\vecx_k;\rho_k)}_{k\in\K_\rho^{\tt x}}$, $\ell=1,\ldots,m$ and $\set{\mu_j(\vecx_k;\rho_k)}_{k\in\K_\rho^{\tt x}}$, $j=1,\ldots,p$.

In fact,~\eqref{eq_ap:grad_Z_1} can be rewritten as
\begin{equation}
\label{eq_ap:derivative_Z_extended}
\begin{aligned}
& \left(\nabla f(\vecy_{k}^i)+\sum_{\ell=1}^m \nabla g_\ell(\vecy_{k}^i) \lambda_\ell(\vecx_k;\rho_k)+\right. \\
& +\sum_{\ell=1}^m \nabla g_\ell(\vecy_{k}^i) (\lambda_\ell(\vecy_{k}^i;\rho_k,)-\lambda_\ell(\vecx_k;\rho_k))+\sum_{j=1}^p \nabla h_j (\vecy_{k}^i) \mu_j(\vecx_k;\rho_k)+ \\
& \left.+\sum_{j=1}^p \nabla
h_j(\vecy_{k}^i)(\mu_j(\vecy_{k}^i;\rho_k)-\mu_j(\vecx_k;\rho_k))\right)^\top
{\vecd_k^i}
\geq-\frac{\xi(\hat{\alpha}_k^i)}{\hat{\alpha}_k^i}, \quad \forall
i \in \J_k \text{ and } k\in\K_\rho^{\tt x}.
\end{aligned}
\end{equation}
Let
\begin{eqnarray*}
&& \set{a_k^1,\ldots,a_k^m}=\set{\lambda_1(\vecx_k;\rho_k),\ldots,\lambda_m(\vecx_k;\rho_k)},\\
&& \set{a_k^{m+1},\ldots,a_k^{m+p}}=\set{\mu_1(\vecx_k;\rho_k),\ldots,\mu_p(\vecx_k;\rho_k)}.\\
\end{eqnarray*}
Assume, by contradiction, that there exists at least one index
$l\in\set{1,\ldots,m+p}$ such that
\begin{equation}
\label{index_q} \Disp\lim_{k\to+\infty \atop k\in\K_\rho^{\tt x}}
|a_k^l|=+\infty.
\end{equation}
Hence, the {sequences} $\{a_k^i\}$, $i=1,\dots,m+p$, cannot be all convergent to zero.
Then, from Lemma~\ref{lemma1_appendix}, there exists an infinite subset
$\K_\rho^{\tt x,a}\subseteq\K_\rho^{\tt x}$ and an index $s\in\set{1,\ldots,m+p}$ such
that,
\begin{equation}
\label{lim_ai_1} \Disp\lim_{k\to+\infty \atop k\in
\K_\rho^{\tt x,a}} \dfrac{a_k^i}{|a_k^s|}=z_i,\quad |z_i|<+\infty,\quad
i=1,\ldots,m+p
\end{equation}
If there is an unique index $l$ that satisfies~\eqref{index_q},
then $s=l$. If more than one index satisfies the
equation, then $s$ is selected as one of the indexes such that
$\set{a_k^s}_{k\in\K_\rho^{\tt x,a}}$ tends to $+\infty$ faster than the
others. Note also that
\begin{equation}\label{lim_ai_2}
z_s=1,\qquad\text{and}\qquad |a_k^s|\to+\infty.
\end{equation}

Dividing the relation~\eqref{eq_ap:derivative_Z_extended} by
$|a_k^s|$, it follows
\begin{equation}
\label{eq_ap:derivative_Z_extended_divided}
\begin{aligned}
& \left(\dfrac{\nabla f(\vecy_{k}^i)}{|a_k^s|}+\sum_{\ell=1}^m \dfrac{\nabla g_\ell(\vecy_{k}^i) a_k^\ell}{|a_k^s|}\right. \\
& +\sum_{\ell=1}^m \nabla g_\ell(\vecy_{k}^i)\dfrac{\lambda_\ell(\vecy_{k}^i;\rho_k)-\lambda_\ell(\vecx_k;\rho_k)}{|a_k^s|}+\sum_{j=1}^p \dfrac{\nabla h_j (\vecy_{k}^i) a_k^{m+j}}{|a_k^s|} \\
& \left.+\sum_{j=1}^p \nabla
h_j(\vecy_{k}^i)\dfrac{\mu_j(\vecy_{k}^i;\rho_k)-\mu_j(\vecx_k;\rho_k)}{|a_k^s|}\right)^\top
{\vecd_k^i}
\geq-\frac{\xi(\hat{\alpha}_k^i)}{\hat{\alpha}_k^i |a_k^s|}, \quad
\forall i\in \J_k \text{ and } k\in\K_\rho^{\tt x,a}.
\end{aligned}
\end{equation}

Reasoning as for~\eqref{eq:grad_Z}, by Proposition~\ref{prop:epsaccurate_estimate} a subsequence $\K_\rho^{\tt x,a,D}\subseteq\K_\rho^{\tt x,a}$ can be extracted such that for all $k\in\K_\rho^{\tt x,a,D}$ it holds $\J_k = \J^*$, $\vecd_k^i=\bar{\vecd}^i$ for all $i\in\J^*$, and $\Dd_k=\Dd^*=\set{\bar{\vecd}^i}_{i\in\J^*}$. Since $\lim_{k\in\K_\rho^{\tt x,a,D}}\vecx_k = \vecx^*$, by Assumption~\ref{condition:geometry_ass} and Proposition~\ref{prop:epsaccurate_estimate}, and using $\varepsilon\in(0,\min\set{\bar{\varepsilon},\varepsilon^*}]$, for sufficiently large $k\in\K_\rho^{\tt x,a,D}$ we have $\T_\X(\vecx^*)=\T_\X(\vecx_k,\varepsilon) =\cone(\Dd_k\cap \T_\X(\vecx_k,\varepsilon))=\cone(\Dd^*)$.

Taking the limit for $k\to+\infty$ and $k\in\K_\rho^{\tt x,a,D}$, and using~\eqref{eq_ap:lim_lambda},~\eqref{eq_ap:lim_mu}, and~\eqref{lim_ai_1}, we obtain
\begin{equation}
\label{eq_ap:lim_derivative_Z_divided} \left(\sum_{\ell=1}^m
z_\ell\nabla g_\ell(\vecx^*)+\sum_{j=1}^p z_{m+j} \nabla
h_j(\vecx^*)\right)^\top \bar{\vecd}^i\geq 0,\quad \forall \bar{\vecd}^i\in\Dd^*.
\end{equation}

Let us recall that $\vecx^*$ satisfies the MFCQ conditions. Let $\vecd$ be the
direction satisfying condition (b) of Definition~\ref{MFCQ}. For every $\vecd\in \T_\X(\vecx^*)$, there
exist nonnegative numbers $\beta_i$ such that
\begin{equation}
\label{eq_ap:direction} \vecd = \sum_{\bar{\vecd}^i\in \Dd^*} \beta_i \bar{\vecd}^i.
\end{equation}
Thus, from~\eqref{eq_ap:lim_derivative_Z_divided}
and~\eqref{eq_ap:direction}, we obtain
\begin{eqnarray}
\label{eq_ap:lim_derivative_Z_divided_2} \hspace{-0.5cm} &&
\left(\sum_{\ell=1}^m z_\ell\nabla g_\ell(\vecx^*)+\sum_{j=1}^p z_{m+j} \nabla
h_j(\vecx^*)\right)^\top \vecd= \\ &&= \sum_{\bar{\vecd}^i\in \Dd^*}
\beta_i\left(\sum_{\ell=1}^m z_\ell\nabla g_\ell(\vecx^*)+\sum_{j=1}^p z_{m+j}
\nabla h_j(\vecx^*)\right)^\top \bar{\vecd}^i \nonumber
\\&&= \sum_{\bar{\vecd}^i\in \Dd^*} \beta_i \sum_{\ell=1}^m z_\ell\nabla
g_\ell(\vecx^*)^\top \bar{\vecd}^i+\sum_{\bar{\vecd}^i\in \Dd^*} \beta_i \sum_{j=1}^p
z_{m+j} \nabla h_j(\vecx^*)^\top\bar{\vecd}^i  \geq 0.\nonumber
\end{eqnarray}
Considering Definition~\ref{MFCQ}, from~\eqref{eq_ap:lim_derivative_Z_divided_2} it follows
\begin{equation}
\label{eq_ap:lim_derivative_Z_divided_3} \sum_{\ell=1}^m z_\ell\nabla g_\ell(\vecx^*)^\top \vecd  \geq 0.
\end{equation}
Theorem~\ref{convergence_eps} and the definition of $z_\ell$ for $\ell\in\set{1,\ldots,m}$, guarantee \begin{equation}
\label{eq_ap:z_notinI}
z_\ell=0,\quad \text{ for all }\quad \ell\notin \I_+(\vecx^*).
\end{equation}
Since $\vecx^*$ satisfies the MFCQ conditions, \eqref{eq_ap:lim_derivative_Z_divided_3} implies
\begin{equation}
\label{eq_ap:z_inI}
z_\ell=0,\quad \text{ for all } \ell\in \I_+(\vecx^*).
\end{equation}
Applying~\eqref{eq_ap:z_notinI} and~\eqref{eq_ap:z_inI} to~\eqref{eq_ap:lim_derivative_Z_divided} we get
\begin{equation}
\label{eq_ap:last_eq}
\left(\sum_{j=1}^p z_{m+j} \nabla h_j(\vecx^*)\right)^\top {\bar{\vecd}^i}\geq 0,\quad \text{ for all }\quad {\bar{\vecd}^i}\in D^*,
\end{equation}
and again using Definition~\ref{MFCQ} and~\eqref{eq_ap:last_eq}, we obtain
\begin{equation}
\label{eq_ap:z_j}
z_{m+j}=0,\quad \text{ for all }\quad j\in\set{1,\ldots,p}.
\end{equation}
In conclusion, we get~\eqref{eq_ap:z_notinI},~\eqref{eq_ap:z_inI}, and~\eqref{eq_ap:z_j}, contradicting~\eqref{lim_ai_2} and this concludes the proof.
\end{proof}

{\small \printbibliography}


\end{document}